\documentclass[a4paper]{article}

\usepackage[left=3cm, top=3cm,bottom=3cm,right=3cm]{geometry}

\usepackage[utf8]{inputenc}
\usepackage{lmodern}
\usepackage{amssymb,amsfonts,amsthm,amsmath,bbm,mathrsfs,aliascnt,mathtools}
\usepackage[english]{babel}
\usepackage[colorlinks]{hyperref}
\usepackage{tikz}
\usepackage{bbm}
\usetikzlibrary{decorations}
\usetikzlibrary{decorations.pathreplacing}
\tikzset{individu/.style={draw,thick}}

\newcommand{\mmp}{\mathbb{P}}
\newcommand{\me}{\mathbb{E}}
\newcommand{\mn}{\mathbb{N}}
\newcommand{\mr}{\mathbb{R}}
\newcommand{\dod}{\overset{{\rm d}}{\to}}
\DeclareMathOperator{\1}{\mathbbm{1}}

\theoremstyle{plain}
\newtheorem{theorem}{Theorem}[section]
\newtheorem{corollary}[theorem]{Corollary}
\newtheorem{lemma}[theorem]{Lemma}

\theoremstyle{definition}

\theoremstyle{remark}
\newtheorem{remark}[theorem]{Remark}

\numberwithin{equation}{section}

\newcommand{\N}{\mathbb{N}}
\newcommand{\Z}{\mathbb{Z}}
\newcommand{\R}{\mathbb{R}}

\DeclareMathOperator{\F}{\mathcal{F}}

\newcommand{\cZ}{\mathcal{Z}}

\renewcommand{\S}{\mathcal{S}}

\newcommand{\ceil}[1]{{\left\lceil #1 \right\rceil}}

\renewcommand{\bar}[1]{\overline{#1}}
\renewcommand{\tilde}[1]{\widetilde{#1}}
\renewcommand{\hat}[1]{\widehat{#1}}

\newcommand{\dd}{\mathrm{d}}

\DeclareMathOperator{\E}{\mathbb{E}}

\renewcommand{\P}{\mathbb{P}}

\newcommand{\I}{\mathcal{I}}

\usepackage{xcolor}

\renewcommand{\rho}{\varrho}
\renewcommand{\epsilon}{\varepsilon}

\title{On the derivative martingale in a branching random walk}
\author{Dariusz Buraczewski,\footnote{Mathematical Institute, University of
Wroc{\l}aw, Wroc{\l}aw, Poland; e-mail: dbura@math.uni.wroc.pl} \
\ Alexander Iksanov\footnote{Faculty of Computer Science and
Cybernetics, Taras Shevchenko National University of Kyiv, Kyiv,
Ukraine;\  e-mail: iksan@univ.kiev.ua} \ \ and \ \ Bastien
Mallein\footnote{Université Sorbonne Paris Nord, LAGA, UMR 7539, Villetaneuse, France; e-mail: mallein@math.univ-paris13.fr}}
\date{\today}

\begin{document}

\maketitle

\begin{abstract}
We work under the A\"{\i}d\'{e}kon-Chen conditions which ensure that the derivative martingale in a supercritical branching random walk on the line converges almost surely to a nondegenerate nonnegative random variable that we denote by $Z$. It is shown that $\me Z\1_{\{Z\le x\}}=\log x+o(\log x)$ as $x\to\infty$. Also, we provide necessary and sufficient conditions under which $\me Z\1_{\{Z\le x\}}=\log x+{\rm const}+o(1)$ as $x\to\infty$. This more precise asymptotics is a key tool for proving distributional limit theorems which quantify the rate of convergence of the derivative martingale to its limit $Z$. The methodological novelty of the present paper is a three terms representation of a subharmonic function of at most linear growth for a killed centered random walk of finite variance. This yields the aforementioned asymptotics and should also be applicable to other models.
\end{abstract}



\section{Introduction: a branching random walk and the derivative martingale}\label{intro}

We consider a discrete-time supercritical {\it branching random walk} (BRW) on
the real line $\mr$. The distribution of the branching random walk
is governed by a point process $\cZ:=\sum_{j=1}^N \delta_{X_j}$ on
$\mr$. The number of offspring, $N = \cZ(\R)$, is a random
variable taking values in $\N_0 \cup \{+\infty\} :=
\{0,1,2,\ldots\} \cup \{+\infty\}$.

It is convenient to associate the evolution of BRW with that of some population of individuals. At time $0$, the population
starts with one individual, the ancestor, which resides at the
origin. At time $1$, the ancestor dies and simultaneously places
offspring on the real line with positions given by the points of
the point process $\cZ$. The offspring of the ancestor form the
first generation of the underlying population. At time $2$, each
particle of the first generation dies and has offspring with
positions relative to their parent's position given by an
independent copy of $\cZ$. The individuals produced by the first
generation particles form the second generation of the population,
and so on.

More formally, let $\I=\bigcup_{n\geq 0}\N^n$ be the set of all
possible individuals. The ancestor label is the empty word
$\varnothing$, its position is $S(\varnothing)=0$. On some
probability space let $(\cZ(u))_{u \in \I}$ be a family of
independent copies of the point process $\cZ$. An individual of
the $n$th generation with label $u = u_1\ldots u_n$ and position
$S(u)$ produces a random number $N(u)$ of offspring at time $n+1$.
The offspring of the individual $u$ are placed at random locations
on $\mr$ given by the positions of the point process
\begin{equation*}
\delta_{S(u)} * \cZ(u) = \sum_{j=1}^{N(u)} \delta_{S(u) + X_j(u)},
\end{equation*}
where $\cZ(u) = \sum_{j=1}^{N(u)} \delta_{X_j(u)}$ and $N(u)$ is
the number of points in $\cZ(u)$. The offspring of the individual
$u$ are enumerated by $uj = u_1 \ldots u_n j$, where
$j=1,\ldots,N(u)$ (if $N(u)<\infty$) or $j=1,2,\ldots$ (if
$N(u)=\infty$), and the positions of the offspring are denoted by
$S(uj)$. No assumptions are imposed on the dependence structure of
the random variables $N(u), X_1(u),X_2(u),\ldots$ for fixed
$u\in\I$. The point process of the positions of the $n$th
generation individuals will be denoted by $\cZ_n$ so that
$\cZ_0=\delta_0$ and
\begin{equation*}
\cZ_{n+1} = \sum_{|u|=n} \sum_{j=1}^{N(u)} \delta_{S(u)+X_j(u)}= \sum_{|u|=n}\sum_{j=1}^{N(u)} \delta_{S(uj)},\quad n\in\mn_0.
\end{equation*}
Here and hereafter, $|u|=n$ means that the sum is taken over all
individuals of the $n$th generation rather than over all
$u\in\N^n$. The sequence of point processes $(\cZ_n)_{n \in
\mn_0}$ is then called a branching random walk.
Throughout the article, we assume that $\me N\in (1, \infty]$
(supercriticality) which implies that the population survives with
positive probability. Notice that the sequence of generation sizes
in the BRW forms a Galton-Watson process provided that $N<\infty$
almost surely (a.s.).

In what follows we always assume that
\begin{equation}\label{eq:m(1)=1}
\me \sum_{i=1}^N e^{-X_i}~=~ 1.
\end{equation}
On the other hand, the situation is not excluded that $\me \sum_{i=1}^N e^{-\gamma X_i}=\infty$ for all $\gamma\neq 1$. Put
\begin{equation*}\label{eq:W_n}
W_n ~:=~ \sum_{|u|=n} e^{-S(u)}, \ \ n\in\mn_0
\end{equation*}
and let $\F_n$ be the $\sigma$-algebra generated by the first $n$ generations, that is, $\F_n = \sigma(\cZ(u):\,
|u|<n)$ where $|u|<n$ means that $u \in \N^k$ for some $k<n$. It is a straightforward consequence of \eqref{eq:m(1)=1} and the branching property that the sequence $(W_n, \F_n)_{n \in\N_0}$ is a
nonnegative martingale and thus converges a.s.~to a random
variable that we denote by $W$. This martingale is called \emph{additive} or \emph{Biggins' martingale}.

In addition to \eqref{eq:m(1)=1} we shall assume that
\begin{equation}\label{eqn:boundarycase}
\E \sum_{i=1}^N e^{-X_i}X_i=0
\end{equation}
which means that we are focussed on the so called {\it boundary
case}. Observe that, under \eqref{eqn:boundarycase}, we have $W=0$
a.s. (see, for instance Theorem on p.~218 in \cite{Lyons:1997}).
Putting
\begin{equation*}
Z_n ~:=~ \sum_{|u|=n} e^{-S(u)}S(u), \ \ n\in\mn_0,
\end{equation*}
we obtain another martingale $(Z_n, \F_n)_{n \in\N_0}$ which is
known in the literature as {\it derivative martingale}. Let ${\rm i}:=\sqrt{-1}$ and $\gamma\in\mr$. Differentiating
formally $\sum_{|u|=n} e^{-(1-{\rm i}\gamma)S(u)}/\me \sum_{|u|=n}
e^{-(1-{\rm i}\gamma)S(u)}$ in $\gamma$ and putting $\gamma=0$
yields ${\rm i} Z_n$ which justifies the term `derivative
martingale'.

Put $$\widetilde W_1:= \sum_{i=1}^N e^{-X_i}(X_i)_+.$$ Here and hereafter, we use the standard notation: for
$x\in\mr$, $x_+:=x\vee 0$, $x_-:=(-x)\vee 0$ and $\log_+x:=\log
(x\vee 1)$. It is known (see Proposition A.3 (iii) in \cite{Aidekon:2013})
that the a.s.\ limit $Z:=\lim_{n\to\infty}Z_n$ exists and is
nonnegative and nondegenerate, that is, $\mmp\{Z>0\}>0$ provided
that conditions \eqref{eq:m(1)=1}, \eqref{eqn:boundarycase},
\begin{equation} \label{eqn:variance}
\sigma^2:=\E \sum_{i=1}^N e^{-X_i}X_i^2<\infty
\end{equation}
and
\begin{equation}\label{eqn:integrability}
\E W_1 (\log_+ W_1)^2 + \E \widetilde W_1 \log_+ \widetilde W_1<\infty 
\end{equation}
hold.   Further, according to Theorem 1.1 in \cite{Chen:2015}, under \eqref{eq:m(1)=1},
\eqref{eqn:boundarycase} and \eqref{eqn:variance}, condition
\eqref{eqn:integrability} is also necessary for the existence of
$Z\geq 0$ which is positive with positive probability.

In some of our main results we shall assume that the
distribution of the displacements of the BRW is nonarithmetic, that is, for all $\delta > 0$,
\begin{equation}\label{eq:nonarithmetic}
\P\{\mathcal{Z}(\R\backslash \delta \Z)>0\}>0,
\end{equation}
where $\Z$ is the set of integers.

Conditions \eqref{eqn:variance} and \eqref{eqn:integrability} are standard assumptions which are imposed in articles dealing with the derivative martingale, see, for instance, \cite{Aidekon:2013,Aidekon+Shi:2014,Chen:2015}. The additional assumption \eqref{eq:nonarithmetic} is often needed for proving distributional convergence or convergence of moments, see \cite{Aidekon:2013} for an analysis of the maximal displacement in a BRW. Conditions \eqref{eq:m(1)=1}, \eqref{eqn:boundarycase},
\eqref{eqn:variance}, \eqref{eqn:integrability} are our standing
assumptions throughout the paper, sometimes referred to thereafter
as Condition $\mathcal{S}$. Condition $\mathcal{S}$ in conjunction with the nonarithmeticity assumption \eqref{eq:nonarithmetic} will be called Condition $\mathcal{S}_{{\rm na}}$.

\section{Main results}

\subsection{Tail behavior of the derivative martingale limit}\label{tail}

Our purpose is to provide a two terms asymptotic expansion for
$\me Z\1_{\{Z\leq x\}}$ as $x\to\infty$. While investigating the
relevant literature we have realized that even the first order
asymptotics of that expectation is not given under optimal
assumptions. Thus, we start by filling up this gap.
\begin{theorem} \label{expa0}
Assume that Condition $\mathcal{S}$ holds. Then
\begin{equation}\label{limit2}
\me Z\1_{\{Z\leq x\}}~\sim~ \log x,\quad x\to\infty.
\end{equation}
\end{theorem}

To formulate our main result, put
\begin{equation}\label{eq:sob4}
W_1^+ :=\sum_{i=1}^N e^{-X_i}\1_{\{X_i\geq 0\}}, \quad
W_1^- :=\sum_{i=1}^N e^{-X_i}\1_{\{X_i<0\}}
\end{equation}
and $X_{\min}:= \min_{1\leq i\le N} X_i$, so that, $X_{\min}$ is the position of the leftmost individual in the first generation. Further, we introduce the following conditions
\begin{equation}\label{eq:2.4}
\E W^+_1 (\log_+ W^+_1)^3+ \E\widetilde W_1 (\log_+ \widetilde W_1)^2<\infty;
\end{equation}
\begin{equation}\label{eq:2.5}
\E W_1^-(\log W_1^-)^3\1_{\big\{\sum_{i=1}^N (1+X_i - X_{\min}) e^{X_{\min} - X_i}\1_{\{X_i<0\}} >C_0 \big\}}<\infty \quad
\mbox{for some } C_0>0
\end{equation}
and
\begin{equation}\label{mom3}
\me \sum_{i=1}^N e^{-X_i}(X_i)_-^3<\infty.
\end{equation}
In what follows, we refer to the union of \eqref{eq:2.4}, \eqref{eq:2.5} and \eqref{mom3} as Condition $\mathcal{S}^\ast$.
\begin{theorem}\label{expa}
Under Condition $\mathcal{S}_{\rm na}$, we have
\begin{equation}\label{limit3}
\me Z\1_{\{Z\leq x\}}=\log x+c+o(1),\quad x\to\infty
\end{equation}
for a finite constant $c$ if, and only if, Condition $\mathcal{S}^\ast$ holds. 
Formula \eqref{limit3} particularly entails
\begin{equation}\label{Z}
\lim_{x\to\infty}x\mmp\{Z>x\}=1.
\end{equation}
\end{theorem}

We proceed with a number of remarks.
\begin{remark}
1) We start by giving one particular example in which condition \eqref{eq:2.5} holds true. Assume that the number of the first generation individuals positioned on the negative halfline is bounded a.s., that is , $\sum_{i=1}^N \1_{\{X_i<0\}}\leq C_0$ a.s.\ for some $C_0>0$. Then $\sum_{i=1}^N (1+X_i - X_{\min}) e^{X_{\min} - X_i}\1_{\{X_i<0\}}\leq C_0$ a.s.\ which entails \eqref{eq:2.5}. Of course, if $\sum_{i=1}^N \1_{\{X_i<0\}}=0$ a.s., then \eqref{eq:2.5} holds trivially.

\noindent 2) A sufficient condition for \eqref{limit3} is
\begin{equation*}
\E W_1 (\log_+ W_1)^3+ \E \widetilde W_1 (\log_+ \widetilde W_1)^2< \infty.
\end{equation*}
Observe that it has a form similar to \eqref{eqn:integrability}.

\noindent 3) In a frequently encountered and mathematically tractable setting, the random variables $X_1$, $X_2,\ldots$ (displacements) are independent and identically distributed and also independent of $N$ (the number of offspring). Direct calculation reveals that Conditions $\S_{\rm na}$ and $\mathcal{S}^*$ are ensured by
\begin{equation*}
\E N\in (1,\infty), \quad \E N (\log_+ N)^2 < \infty;
\end{equation*}
\begin{equation*}
\E e^{-X_1}=(\E N)^{-1},\quad \E e^{-X_1}X_1 = 0,\quad \E e^{-X_1}X_1^2 < \infty;
\end{equation*}
\begin{equation*}
\text{the distribution of }~ X_1~ \text{is nonarithmetic}
\end{equation*}
and
\begin{equation}\label{Mail}
\E N (\log_+N)^3 < \infty,\quad \E e^{-X_1}(X_1)_-^3 < \infty,
\end{equation}
respectively. Alternatively, but a bit informally, this can be seen by identifying the $n$th generation of the BRW described above with the $(n+1)$st generation of a BRW driven by a point process $\cZ^\ast:=N\delta_{X_1}$ (the correspondence is set by replacing the position of each parent in the latter BRW with the position of its children). Thus, neglecting the numbering of generations one may replace, for instance, the condition $\me W_1(\log_+ W_1)^2<\infty$ which is a part of \eqref{eqn:integrability} with $\me Ne^{-X_1}(\log_+ Ne^{-X_1})^2<\infty$. The latter is equivalent to $\me N(\log_+ N)^2<\infty$ and $\me e^{-X_1}(X_1)_-^2<\infty$.
\end{remark}

\subsection{The rate of convergence of the derivative martingale to its limit}\label{rate}

Recall that the characteristic function of a general nondegenerate $1$-stable distribution $\nu$ takes the form $$t\mapsto \exp({\rm i}at-b|t|(1+{\rm i}\beta {\rm sgn}\,t (2/\pi)\log |t|)),\quad t\in\mr,$$ where $a\in\mr$, $b>0$ and $\beta\in\mr$, $|\beta|\leq 1$, and that $\nu$ is uniquely determined by the generating triple $(a,b,\beta)$. The L\'{e}vy spectral function $M^\ast$ of $\nu$ is given by $M^\ast(x)=b_1|x|^{-1}$ for $x<0$ and $M^\ast(x)=-b_2 x^{-1}$ for $x>0$, where $b_1, b_2\geq 0$ satisfy $b=(b_1+b_2)\pi/2$ and $\beta=(b_2-b_1)/(b_2+b_1)$. When $b_1=0$, $b_2>0$, so that $\beta=1$ the distribution $\nu$ is called {\it spectrally positive}.

As an application of Theorem \ref{expa} which is a result on the tail behavior of $Z$ we state a one-dimensional limit theorem. Set $\F_\infty:=\sigma(\F_n: n \in \mn_0)$ and note that $Z$, the a.s.\ limit of $Z_n$, is an $\F_\infty$-measurable random variable. As usual, $\overset{\mmp}{\to}$ and $\dod$ will denote convergence in probability and in distribution, respectively.
\begin{theorem}\label{main2}
Assume that Conditions $\mathcal{S}_{\rm na}$ and $\mathcal{S}^\ast$ hold.
Then, for every bounded continuous function $f:\mr \to \mr$,
\begin{equation}\label{limit_main01}
\me\big(f(n^{1/2}(Z-Z_n+(2^{-1}\log n) W_n)\big)\big |\F_n\big)~\overset{\mmp}{\to}~ \me(f(ZL)|\F_\infty),\quad n\to\infty,
\end{equation}
which particularly entails
\begin{equation}\label{limit_inter2}
n^{1/2}(Z-Z_n+(2^{-1}\log n) W_n)~\dod~ ZL,\quad n\to\infty.
\end{equation}
Here, a random variable $L$ is assumed independent of $\F_\infty$ and has a $1$-stable distribution with the generating triple $((c+1-\gamma)(2/(\pi\sigma^2))^{1/2}, (\pi/(2\sigma^2))^{1/2}, 1)$, $\gamma$ is the Euler-Mascheroni constant, and  $c$ is the same constant as in \eqref{limit3}. Thus, the distribution of $L$ is spectrally positive with characteristic function
$$\me e^{{\rm i}tL}=\exp\big({\rm
i}(c+1-\gamma)(2/(\pi\sigma^2))^{1/2}t-(\pi/(2\sigma^2))^{1/2}|t|(1+{\rm i}\,{\rm
sgn}\,(t)(2/\pi)\log |t|)\big),\quad t\in\mr.$$
\end{theorem}

Plainly, Theorem \ref{main2} is a result on the rate of convergence of the derivative martingale to its a.s.\ limit.
\begin{remark}
Mimicking the proof of Theorem \ref{main2} one can also show that, for every bounded continuous function $f:\mr \to \mr$, on the set
of survival $\{\cZ_{n}(\mr)>0\quad\text{for all}~n\in\mn\}$,
\begin{equation}\label{limit_main00}
\me\Big(f\Big(\frac{n^{1/2}}{Z_n}(Z-Z_n+(2^{-1}\log n) W_n)\Big)\Big)\Big|\F_n\Big)~\overset{\mmp}{\to}~ \E f(L),\quad n\to\infty.
\end{equation}
As a consequence, a counterpart of \eqref{limit_inter2} holds, namely, conditionally on the survival,
\begin{equation}\label{limit_main22}
\frac{n^{1/2}}{Z_n}(Z-Z_n+(2^{-1}\log n) W_n)~\dod~
L,\quad n\to\infty.
\end{equation}
We omit further details.
\end{remark}

The rest of the article is structured as follows. In Section \ref{appr} we explain our approach which is based on a novel look at a Poisson equation on the halfline. Also in the section is a brief survey of some earlier papers dealing with a general Poisson equation. In Section \ref{comp} we compare our results to similar ones available in the literature. In Section \ref{nota} we introduce a standard random walk associated with the BRW and lay down the frequently used notation. In Section \ref{subharm} which is the core of our work we prove a representation of subharmonic functions of at most linear growth for killed centered standard random walks with finite variance. As a corollary, we show that actually such functions grow linearly. While Theorems \ref{expa0} and \ref{expa} are proved in Section \ref{proofs1}, Theorem \ref{main2} is proved in Section \ref{proofs2}. The appendix collects several Abelian and Tauberian theorems related to the de Haan class of slowly varying functions and some auxiliary facts about standard random walks, Lebesgue integrable and directly Riemann integrable functions.

\section{Discussion}
\subsection{Our approach}\label{appr}

To determine the tail behavior of $Z$ we work with its Laplace transform. Formula \eqref{eqn:harmonicD} written in terms of this Laplace transform is an instance of a Poisson equation. In view of this, our principal purpose is to develop an approach towards understanding the asymptotics of solutions to a {\it general} Poisson equation
\begin{equation}\label{eq:Poisson equation}
K(x) =\me K(x+\eta)-L(x),\quad x\in\mr,
\end{equation}
where $\eta$ is a random variable and $L:\mr\to\mr$ is a given function. Especially, we are interested in situations in which $K$ exhibits a linear growth.

When $\me\eta\neq 0$ and $\me |\eta|<\infty$, \eqref{eq:Poisson equation} is called {\it renewal equation}. In this case, $$K(x)=-\int_\mr L(x+y)U^\ast({\rm d}y),\quad x\in\mr,$$ where, with $\eta_1$, $\eta_2,\ldots$ being independent copies of $\eta$, $U^\ast$ is the (locally finite) renewal measure defined by $U^\ast({\rm d}y)=\sum_{k\geq 0}\mmp\{\eta_1+\ldots+\eta_k\in {\rm d}y\}$. Furthermore, the asymptotics of $K$ is well-understood and driven by the key renewal theorem in which case $$\lim_{x\to \pm \infty}K(x)=-(\me \eta)^{-1}\int_\mr L(y){\rm d}y$$ (depending on the sign of $\me\eta$ the limit is as $x\to-\infty$ or $x\to+\infty$) or its relatives, see, for instance, Section 6.2 in \cite{Iksanov:2016}.

In this article our focus is on the centered case $\me\eta=0$ in which the renewal measure (potential) $U^\ast$ is not locally finite. This makes things more complicated, and one has to find a proper replacement for $U^\ast$. This task was accomplished by Spitzer in Section 28 of \cite{Spitzer:2001} for centered random walks on integers and then by Port and Stone in \cite{Port+Stone:1969} in a general setting. Assuming that the distribution of $\eta$ is spread-out (that is, some convolution power of it has a nontrivial absolutely continuous component) and that $L$ is a bounded function of compact support these authors proposed a limiting procedure yielding the potential kernel $A$ defined by
\begin{equation*}
AL(x): =\int_\mr L(x-y)a(y){\rm d}y-\int_\mr L(x-y)\rho({\rm d}y)+ b\int_\mr L(y){\rm d}y-L(x),\quad x\in\mr.
\end{equation*}
Here, $a:\mr\to\mr$ is a continuous function satisfying $\lim_{x\to \pm \infty}
(a(x-y) - a(x)) = \mp {\tt s}^{-2}y$, where ${\tt s}^2 = \E\eta^2$; $\rho$ is a finite measure and $b$ is a constant. As a consequence, it was shown in Theorem 10.3 of \cite{Port+Stone:1969} that any positive (or more generally
bounded from below) solution to \eqref{eq:Poisson equation} is of the form
\begin{equation}\label{eq:Poisson equation solution}
K(x) = AL(x) + \Big(c{\tt s}^{-2}\int_\mr L(y){\rm d}y\Big) x+ d,\quad x\in\mr,
\end{equation}
where $d$ is any constant and $|c| \le 1$. It is known that either $K(x)$ converges to a positive constant or behaves linearly as $x\to\infty$ depending on whether $\int_\mr L(y){\rm d}y$ is zero or not. There is an extension of the results discussed above to the case where $L$ is not necessarily compactly supported and rather satisfies an integral condition, see Theorem
3.1 in \cite{Brofferio+Buraczewski+Damek:2012} or Theorem 3.2 in \cite{Buraczewski:2007}.

While investigating a particular Poisson equation related to a smoothing transform (see the beginning of Section \ref{decomp} for the definition and some
more details) Durrett and Liggett in \cite{Durrett+Liggett:1983} were concerned
with the asymptotic behavior of a given solution to
\eqref{eq:Poisson equation} rather than in description of the set
of all solutions. These authors invented a novel approach based on
Feller's duality principle (Lemma 1 on p.~609 in
\cite{Feller:1971}). This enabled them to employ the key renewal
theorem for describing the asymptotic behavior of the given
solution. In a more general setting similar ideas were exploited
by Liu in \cite{Liu:1998}.

The main methodological achievement of the present work is an
explicit formula, other than \eqref{eq:Poisson equation solution}, for solutions of at most linear growth to a
Poisson equation on the halfline. Among other things this provides a way to easily obtain the
precise asymptotic behavior of those solutions. Roughly speaking,
the idea is as follows. We are interested in the asymptotics of a
solution $f$ at $\infty$, so that the values $f(x)$ for $x\leq 0$
should play no role. Thus, we regard $f$ as a solution to a
Dirichlet problem: given the values of $f$ on $(-\infty,0]$ (which
can be thought of as {\it boundary values}) we intend to recover
$f$ on $(0,\infty)$ which is nothing else but a subharmonic
function of at most linear growth for a recurrent standard random
walk killed upon entering $(-\infty, 0]$.

\subsection{Comparison to earlier literature} \label{comp}

\noindent {\sc Comments on Section \ref{tail}}. Theorem \ref{expa0} provides an improvement over Theorem 2.18 in
\cite{Durrett+Liggett:1983} and Theorem 4.2 in \cite{Liu:1998}
obtained for $Z$ being a fixed point of the smoothing transform.
In the former, relation \eqref{limit2} is proved in
the situation that $N\geq 2$ is a deterministic integer, that
conditions \eqref{eq:m(1)=1}, \eqref{eqn:boundarycase} and
\eqref{eq:nonarithmetic} hold, and that $\me W_1^\gamma<\infty$
for some $\gamma>1$. In the latter, while $N$ is random with $\me
N>1$, the other conditions ensuring \eqref{limit2} are comparable
to those in \cite{Durrett+Liggett:1983}.

Theorem \ref{expa} strengthens several results on the tail
behavior of $Z$ available in the literature. The best previously
known sufficient conditions for \eqref{Z} that we are aware of are
in Theorem 1.4 of \cite{Madaule:2018}. In addition to Condition
$\mathcal{S}_{\rm na}$ the author requires $$\me\Big(\sum_{i=1}^N
e^{-X_i}+\sum_{i=1}^N e^{-X_i}(X_i)_+\Big)\log_+\Big(\sum_{i=1}^N e^{-X_i}+\log\sum_{i=1}^N
e^{-X_i}(X_i)_+\Big)^5<\infty.$$ To be more
precise, in the last cited theorem it is claimed that
\begin{equation*}
\lim_{x\to\infty}x\mmp\{Z>x\}=b,
\end{equation*}
where $b$ is the product of two positive constants expressed in
terms of the minimal position of BRW's individuals over the whole
population and the random variable $Z$. Our Theorem \ref{expa}
reveals that $b$ is actually equal to one, thereby giving an explicit relationship between these two constants. Under stronger moment
assumptions a relation like \eqref{Z} was also proved in Theorem
1.2 of \cite{Buraczewski:2009} for $Z$ being a fixed point of the
smoothing transform. Last but not least, a counterpart of \eqref{limit3} in the context of branching Brownian
motion was proved in Proposition 4.1 of \cite{Maillard:2012}. Our condition \eqref{Mail} is reminiscent of Maillard's condition.

\noindent {\sc Comments on Section \ref{rate}}. Limit theorems
providing a rate of convergence have been and still are quite
popular in the area of branching processes. Surveys of the
relevant literature can be found in
\cite{Iksanov+Kolesko+Meiners:2020} and \cite{Maillard+Pain:2019}. The latter article discusses, among others, limit theorems for some models of statistical mechanics. A large selection of rate of convergence results for more complicated branching processes, including branching diffusions and superprocesses, can be traced via the references given in \cite{Ren etal:2019+}.

Theorem \ref{main2} is a counterpart of Proposition 2.1 in \cite{Maillard+Pain:2019} obtained for the derivative martingale which corresponds to a branching Brownian motion. Observing the martingale at nonnegative integer times only yields a particular version of $(Z_n,\F_n)_{n\in\mn_0}$ investigated here, with $\sigma^2=1$. According to Theorem \ref{main2}, the random variable $L$ appearing in \eqref{limit_inter2} has a $1$-stable distribution with the generating triple
$((c+1-\gamma)(2/\pi)^{1/2}, (\pi/2)^{1/2}, 1)$, whereas according to Proposition 2.1 in \cite{Maillard+Pain:2019} the generating triple is $((c-\gamma)(2/\pi)^{1/2}, (\pi/2)^{1/2}, 1)$, that is, $1$ is lost. The error in \cite{Maillard+Pain:2019} is caused by missing the term $x\mmp\{Z>x\}$ which converges to $1$ as $x\to\infty$ in the equality $$\int_0^x \mmp\{Z>y\}{\rm d}y=\me Z\1_{\{Z\leq x\}}+x\mmp\{Z>x\},\quad x>0$$ (see formula (1.9) and Lemma C.1 in \cite{Maillard+Pain:2019}).

\section{A standard random walk associated with BRW}\label{nota}

Under \eqref{eq:m(1)=1}, denote by $\xi$ a random variable with
distribution given by
\begin{equation}\label{xidef}
\me t(\xi)=\me \sum_{i=1}^N e^{-X_i}t(X_i)
\end{equation}
for any measurable bounded function $t:\mr\to\mr^+$, where $\mr^+:=[0,\infty)$. Note that \eqref{xidef} also holds for real-valued $t$ whenever the left- or right-hand side of \eqref{xidef} is well-defined, possibly infinite.

Observe that Condition $\mathcal{S}$ implies that $\me \xi=0$ and $\me\xi^2=\sigma^2<\infty$. Further, we stress that supercriticality in combination with \eqref{eq:m(1)=1} guarantees that $\mmp\{\xi=0\}<1$ (taken together with $\me\xi=0$ the latter means that the distribution of $\xi$ is nondegenerate, whence $\sigma^2>0$). Indeed, assuming the contrary
$$1=\mmp\{\xi=0\}=\me\sum_{i=1}^N e^{-X_i}\1_{\{X_i=0\}}$$ we
conclude that $N=1$ and $X_1=0$ a.s., a contradiction to supercriticality. Additionally, note that Condition $\mathcal{S}_{\rm na}$ implies
that the distribution of $\xi$ is nonarithmetic, that is, concentrated on $d\mathbb{Z}$ for no $d> 0$.

We denote by $S:=(S_n)_{n\in\mn_0}$ a standard random walk defined
by $S_n-S_0:=\xi_1+\ldots+\xi_n$ for $n\in\mn$, where $\xi_1$,
$\xi_2,\ldots$ are independent copies of $\xi$ which are also
independent of $S_0$. For $x\in\mr$, we denote by $\mmp_x$ the
distribution of the random walk $(S_n)_{n\in\mn_0}$ when $S_0=x$
a.s. As usual, we write $\mmp$ for $\mmp_0$.

It is a well-known fact that the behavior of BRW is driven, among
others, by the random walk $S$.
A classical example of this connection is the so-called {\it
many-to-one lemma} which can be traced back at least to Kahane and
Peyrière \cite{Kahane+Peyriere:1976,Peyriere:1974}. We quote it
from Theorem 1.1 in \cite{Shi:2015}.
\begin{lemma}[Many-to-one]\label{many}
For each $n\in\mn$ and a measurable bounded function $t:\mr^n \to\mr^+$,
\[
\E\sum_{|u|=n}e^{-S(u)}t(S(u_1),\ldots, S(u_1\ldots u_n))=\E t(S_1,\ldots, S_n),
\]
where $u = u_1\ldots u_n$.
\end{lemma}

Let $(\tau_k)_{k\in\mn_0}$ be the sequence of weak descending
ladder epochs, defined by $\tau_0:=0$ and, for $k\in\mn$, $\tau_k:=\inf\{j>\tau_{k-1}: S_j\leq S_{\tau_{k-1}}\}$. Also, let $(\sigma_n)_{n\in\mn_0}$ be the sequence of strict ascending ladder epochs, defined by $\sigma_0:=0$ and, for $n\in\mn$, $\sigma_n:=\inf\{i>\tau_{n-1}: S_i>S_{\sigma_{n-1}}\}$.  In view of $\me\xi=0$, all these random variables are a.s.\ finite. Under $\mmp$, $(S_{\tau_k})_{k\in\mn_0}$ and $(S_{\sigma_n})_{n\in\mn_0}$, being the sequences of weak descending and strict ascending ladder heights, form standard random walks with independent
nonpositive and nonnegative jumps having the same distribution as $S_{\tau_1}$ and $S_{\sigma_1}$, respectively. Under $\mmp$, denote by $U$ and $V$ the renewal functions of $(-S_{\tau_k})_{k\in\mn_0}$ and $(S_{\sigma_k})_{n\in\mn_0}$, respectively, that is,
\begin{equation}\label{renewal}
U(x):=\sum_{k\geq 0}\mmp\{-S_{\tau_k}< x\}\quad \text{and}\quad V(x):=\sum_{k\geq 0}\mmp\{S_{\sigma_k}\leq x\},\quad x\in\mr.
\end{equation}
Plainly, $U(x)=V(x)=0$ for $x<0$. Observe that $U$ is a left-continuous renewal function which is a slight digression, for typically renewal functions are defined to be right-continuous. Nevertheless, the so defined $U$ shares all the standard asymptotic properties of right-continuous renewal functions.

\section{Subharmonic functions of at most linear growth for the killed random walk}\label{subharm}

Throughout this section we retain the notation $S:=(S_n)_{n\in\mn_0}$ for a standard random walk, not necessarily related to the BRW. All the other notation introduced in Section \ref{nota} is also retained but associated to the $S$ as above. We shall assume, without further notice, until the end of this section that the distribution of $\xi$ is nondegenerate, that $\me\xi=0$ and $\me\xi^2<\infty$ (the only exception is Lemma \ref{lem:renewIntegral} in which finiteness of the second moment is not assumed). The following formulae which are ensured by Lemma \ref{don} (a,b) will be often used
\begin{equation}\label{moments}
\mu:=(-\me S_{\tau_1})\in (0,\infty)\quad\text{and}\quad \nu:=\me S_{\sigma_1}\in (0,\infty)
\end{equation}
and
\begin{equation}\label{asymp}
\lim_{x\to\infty}(U(x)/x)=\mu^{-1}\quad \text{and}\quad\lim_{x\to\infty} (V(x)/x)=\nu^{-1}.
\end{equation}

\subsection{Auxiliary results}
Set $\tau:=\inf\{n\in\mn_0: S_n\leq 0\}$ and note that $\tau=\tau_1$ under $\P_x$ for $x>0$ whereas $\tau=0$ under $\P_x$ for $x\leq 0$. For all $x \in \R$, denote by
\[
  \sigma(x): = \inf\{n\in\mn_0: S_n > x\}
\]
the first passage of $S$ into $(x, \infty)$. We now present an alternative formula for the renewal function $U$.
\begin{lemma} \label{lem:renewalRepresentation}
For all $x \geq 0$,
\[
  \mu U(x)= \lim_{y \to \infty} y \P_x\{\sigma(y)<\tau\}=\lim_{y \to \infty} \E_x S_{\sigma(y)}\1_{\{\sigma(y)<\tau\}},
\]
where $\mu = -\E S_{\tau_1}<\infty$.
\end{lemma}
\begin{proof}
These equalities can be found in \cite{Alsmeyer+Mallein:2019+}. Namely, equation (32) there gives, for $x \geq 0$,
\[
  \mu U(x) = x - \E_x S_\tau.
\]
Then, the first equation follows from Corollary 4.4 and equation (35) in Lemma 4.3 (both in the cited article) can be written as
\[
  \lim_{y \to \infty} \E_x (S_{\sigma(y)}-y)\1_{\{\sigma(y)<\tau\}}= 0
\]
which completes the proof of the second equality.
\end{proof}

Lemma \ref{lem:harmonicKilled} is a restatement of Proposition 4.1 in \cite{Alsmeyer+Mallein:2019+} which characterizes
right-continuous functions $f:\mr\to\mr$ satisfying
\begin{equation}
  \label{eqn:harmonicKilled}
  \begin{cases}
    f(x) = \E f(x+\xi)\1_{\{\xi>-x\}}= \E_x f(S_1)\1_{\{S_1>0\}}, & \text{ if } x > 0,\\
    f(x) = 0, &\text{ if } x \leq 0,\\
    \limsup_{x \to \infty}(|f(x)|/x) < \infty.
  \end{cases}
\end{equation}
In words, the so defined $f$ are harmonic functions of at most linear growth for the killed centered random walk with finite variance.

For $d > 0$, we say that the distribution of $\xi$ is $d$-arithmetic if $\P\{\xi \in d \Z\}=1$, and $d$ is the largest number with this property. With a slight abuse of notation, we say that the distribution of $\xi$ is $0$-arithmetic if it is nonarithmetic, and that a function $\kappa(\cdot)$ is $0$-periodic if it is a constant.
\begin{lemma}\label{lem:harmonicKilled}
Assume that the distribution of $\xi$ is $d$-arithmetic for $d\geq 0$. Then if $f$ satisfies \eqref{eqn:harmonicKilled}, there exists a right-continuous $d$-periodic function $\kappa(\cdot)$ such that $f(x)= \kappa(x) U(x)$ for $x>0$.
\end{lemma}
In particular, note that any solution to \eqref{eqn:harmonicKilled} is a scalar multiple of the renewal function $U$ provided that the distribution of $\xi$ is nonarithmetic.

Given next is a formula which represents the expectation of an additive functional of the killed random walk in terms of renewal functions.
\begin{lemma}\label{lem:renewIntegral}
Not assuming that $\me \xi^2<\infty$, for all measurable functions $p:\mr^+\to \mr^+$ and $x>0$,
\[
  \E_x \sum_{k=0}^{\tau-1} p(S_k)=\int_{[0,\,x]} \dd U(y)\int_{[0,\,\infty)}\dd V(z)p(x-y+z),
\]
where $V$ is the renewal function defined in \eqref{renewal}. Here, both sides of the equality may be infinite.
\end{lemma}
\begin{proof}
Set $r(x):=\int_{[0,\,\infty)} p(x+z) \dd V(z)$ for $x \geq 0$. We use a standard decomposition of $S$ into cycles: for $x>0$,
\begin{multline*}
\me_x\sum_{j=0}^{\tau-1}p(S_j)=\me \sum_{j\geq
0}p(x+S_j)\1_{\{x+S_1>0,\ldots, x+S_j>0\}}=\me \sum_{k\geq
0}\sum_{j=\tau_k}^{\tau_{k+1}-1}p(x+S_j)\1_{\{x+S_1>0,\ldots,x+S_j>
0\}}\\= \me \sum_{k\geq
0}\1_{\{x+S_{\tau_k}>0\}}\sum_{j=\tau_k}^{\tau_{k+1}-1}p(x+S_j)=\me
\sum_{k\geq
0}\1_{\{x+S_{\tau_k}>0\}}\sum_{j=0}^{\tau_{k+1}-\tau_k-1}p(x+S_{\tau_k}+(S_j-S_{\tau_k}))\\=
\me \sum_{k\geq
0}\1_{\{-S_{\tau_k}<x\}}r(x-(-S_{\tau_k}))=\int_{[0,\,x]}r(x-y){\rm
d}U(y).
\end{multline*}
Here, the third equality follows from the fact that $0\leq
-S_{\tau_1}\leq -S_{\tau_2}\leq \ldots$ are the weak record values of
the sequence $(-S_j)_{j\in\mn_0}$, whence, for integer $j\in
[\tau_k, \tau_{k+1}-1]$,
$$\1_{\{x+S_1> 0,\ldots, x+S_j>0\}}=\1_{\{x+S_{\tau_1}>0,\ldots, x+S_{\tau_k}>0\}}=\1_{\{x+S_{\tau_k}>0\}}.$$ To
explain the penultimate equality, note that given $S_{\tau_k}$,
for any $y\in\mr$, by the strong Markov property,
$\sum_{j=0}^{\tau_{k+1}-\tau_k-1}p(y+(S_j-S_{\tau_k}))$ has the
same $\P$- distribution as $\sum_{j=0}^{\tau_1-1}p(y+S_j)$ which, in its
turn, has the same $\P$-distribution as $\sum_{k\geq
0}p(y+S_{\sigma_k})$ by the duality principle (see Lemma 1 on
p.~609 in \cite{Feller:1971}). In particular,
$$\me\Big(\sum_{j=0}^{\tau_{k+1}-\tau_k-1}p(x+S_{\tau_k}+(S_j-S_{\tau_k}))\Big|S_{\tau_k}\Big)=r(x+S_{\tau_k}).$$ The proof of Lemma \ref{lem:renewIntegral} is complete.
\end{proof}

\subsection{New results}

In this section, we extend Lemma \ref{lem:harmonicKilled} by
characterizing right-continuous subharmonic functions of at most
linear growth for the killed random walk. More precisely, given
$g:\mr^+\to\mr^+$ a c\`{a}dl\`{a}g function and $h: (-\infty,
0]\to \mr$ a right-continuous bounded function, we aim at finding
all right-continuous functions $f$ that satisfy
\begin{equation}\label{eqn:problem}
  \begin{cases}
    f(x) = \E f(x+\xi)-g(x), &\text{ if } x > 0\\
    f(x) = h(x), &\text{ if } x \leq 0
  \end{cases}
\end{equation}
and
\begin{equation}\label{eqn:asy}
\limsup_{x \to \infty}(|f(x)|/x) < \infty.
\end{equation}
The definition of directly Riemann integrable (dRi) functions
which are mentioned below can be found in Section \ref{sect:dri}.
\begin{theorem}\label{prop:gen}
Assume that the distribution of $\xi$ is $d$-arithmetic for $d\geq 0$. If solutions $f$ to \eqref{eqn:problem} exist, then, for each $x>0$,
\begin{equation}\label{eqn:sumCondition}
\E_x \sum_{k=0}^{\tau-1}g(S_k) < \infty.
\end{equation}
Conversely, if \eqref{eqn:sumCondition} holds for some $x>0$ and
the function $g$ is dRi on $\mr^+$, then there exist solutions $f$ to \eqref{eqn:problem} with $\lim_{x\to\infty}(f(x)/x)=0$. Furthermore, to any solution $f$ satisfying \eqref{eqn:asy} there corresponds a $d$-periodic right-continuous function $\kappa(\cdot)$ such that,
for all $x>0$,
\begin{equation}\label{impo1}
f(x) = \kappa(x)U(x) + \E_x h(S_\tau)-\E_x \sum_{k=0}^{\tau- 1}g(S_k).
\end{equation}
\end{theorem}
\begin{remark}\label{rem_im}
Assume that the distribution of $\xi$ is nonarithmetic (the
arithmetic case is discussed in Remark \ref{arith}). Then so is the distribution of $S_{\sigma_1}$, see Lemma \ref{don}(c). According to
Lemmas \ref{lem:renewIntegral} and \ref{integr}(c), condition
\eqref{eqn:sumCondition} holding for some $x>0$ does not even guarantee that the function
$g$ is Lebesgue integrable on $\mr^+$. However, by Lemma \ref{integr} (d), it does under an additional uniformity condition. Conversely, while by Lemma \ref{integr}(a), \eqref{eqn:sumCondition} may fail to hold for each $x>0$ if $g$ is Lebesgue integrable, by Lemma \ref{integr}(b), direct Riemann integrability of $g$ is a sufficient condition ensuring that \eqref{eqn:sumCondition} holds for each
$x>0$. Summarizing, we think that condition \eqref{eqn:sumCondition} alone is
not sufficient for proving \eqref{imp123}. This is the reason
behind introducing in Theorem \ref{prop:gen} the additional
assumption that $g$ is dRi which in conjunction with
\eqref{eqn:sumCondition} guarantees that \eqref{imp123} holds, see Lemma \ref{dri2}.
\end{remark}
The proof of Theorem \ref{prop:gen} consists of the three
steps. First, in Lemma \ref{lem:firstStep}, we prove that
condition \eqref{eqn:sumCondition} is necessary for the existence
of a solution. Second, in Lemma \ref{lem:simple}, we exhibit a
particular solution to \eqref{eqn:problem} which is a subharmonic
function of sublinear growth. Third, in the proof of Theorem \ref{prop:gen}, using the linearity of \eqref{eqn:problem} we show
that any solution to \eqref{eqn:problem} is the sum of a harmonic
function of linear growth and the subharmonic function obtained at
the second step.
\begin{lemma}\label{lem:firstStep}
Assume that condition \eqref{eqn:sumCondition} does not hold for $x=x_0>0$. Then no solution to \eqref{eqn:problem} exists.
\end{lemma}
\begin{proof}
Assume on the contrary that there exists a solution to \eqref{eqn:problem} and denote it by $f$. We define $g$ for negative arguments by $g(x):=\E f(x+\xi) - h(x)$, $x\le 0$.
For $n\in\mn_0$, put
$M_n:=f(S_n)-\sum_{k=0}^{n-1}g(S_k)$ and, for $n\in\mn$, let $\mathcal{G}_n$
denote the $\sigma$-algebra generated by $\xi_1,\ldots, \xi_n$, with $\mathcal{G}_0$ being the trivial $\sigma$-algebra.
The sequence $(M_n, \mathcal{G}_n)_{n\in\mn_0}$ is a
$\mmp_{x_0}$-martingale. Since, for each $n\in\mn_0$, $\tau\wedge \sigma(y)\wedge n$ is a stopping time with respect to the filtration $(\mathcal{G}_k)_{k\in\mn}$, then, for $y\geq x_0$, the sequence
$(M_{\tau\wedge \sigma(y)\wedge n}, \mathcal{G}_n)_{n\in\mn_0}$ is also a $\mmp_{x_0}$-martingale. In particular,
\begin{equation} \label{eqn:stoppedMart}
f(x_0) = \E_{x_0} M_0 = \E_{x_0} M_{\tau\wedge \sigma(y)\wedge n},\quad n\in\mn_0.
\end{equation}

We intend to show that
\begin{equation}\label{equa1}
\lim_{n\to\infty} \E_{x_0} M_{\tau\wedge \sigma(y)\wedge n}=\E_{x_0} M_{\tau\wedge \sigma(y)}.
\end{equation}
Note that $\lim_{n\to\infty} M_{\tau\wedge \sigma(y)\wedge n}=M_{\tau\wedge \sigma(y)}$ $\mmp_{x_0}$-a.s. Hence, according to the Lebesgue dominated convergence theorem, it is enough to check that
\begin{equation}\label{ineq1}
\me_{x_0}\sup_{n \geq 0} |M_{\tau\wedge \sigma(y)\wedge n}|<\infty.
\end{equation}
To this end, write, for $n\in\mn_0$,
\begin{multline*}
|M_{\tau\wedge \sigma(y)\wedge n}|\leq |f(S_n)|\1_{\{\tau\wedge \sigma(y)>n\}}+|f(S_{\tau\wedge \sigma(y)})|\1_{\{\tau\wedge \sigma(y)\leq n\}}+\sum_{k=0}^{\tau\wedge \sigma(y)-1}g(S_k)\\ \leq \sup_{z \in [0,\,y]}|f(z)|+ |h(S_\tau)|+|f(S_{\sigma(y)})|+(\tau\wedge\sigma(y))\sup_{z \in [0,\,y]} g(z)\quad \mmp_{x_0}-\text{a.s.}
\end{multline*}
having utilized the fact that $S_k\in (0,y]$ on the event $\{\tau\wedge \sigma(y)>k\}$ for the first and the last summands, and $f(x)=h(x)$ for $x\leq 0$ in combination with $S_\tau\leq 0$ $\mmp_{x_0}$-a.s.\ for the second summand. To prove inequality \eqref{ineq1} we have to show that the right-hand side of the last centered formula (which does not depend on $n$) is $\mmp_{x_0}$-integrable.

Since $h$ is bounded on $(-\infty, 0]$ by assumption and $S_\tau\leq 0$ $\mmp_{x_0}$-a.s., we trivially infer $\me_{x_0}|h(S_\tau)|<\infty$. Further, since, by assumption, $f$ is a right-continuous function of at most linear growth, there exists $C>0$ such that $|f(z)| \leq C(z+1)$ for $z\geq 0$. Hence, $\sup_{z \in [0,\,y]}|f(z)|<\infty$, and also $\me_{x_0}|f(S_{\sigma(y)})|\leq C(\me_{x_0}S_{\sigma(y)}+1)=C(\nu V(y-x_0)+1)<\infty$, where $V$ is the renewal function defined in \eqref{renewal}. The last inequality is justified by \eqref{moments}. The inequality $\sup_{z \in [0,\,y]} g(z)<\infty$ is secured by our assumption that $g$ is a c\`{a}dl\`{a}g function. So, it remains to prove that
\begin{equation}\label{ineq2}
\me_{x_0}(\tau\wedge \sigma(y))<\infty.
\end{equation}
Since the distribution of $\xi$ is nondegenerate, there exists $\delta > 0$ such that $\P\{\xi > \delta\}\in (0,1)$. Set $N_y:=\ceil{y/\delta}$,
where $z\mapsto \ceil{z}$ for $z\in\mr$ is the ceiling function. Then
\[\sup_{z \in [0,\,y]} \P_z\{\tau \wedge \sigma(y)\leq N_y\}\geq \sup_{z \in [0,\,y]} \P_z\{\inf_{1\leq j\leq N_y}\,\xi_j>\delta\}=(\P\{\xi > \delta\})^{N_y}
=:\rho_y\in (0,1).\] Now an application of the Markov property yields, for $k\in\mn$,
\[\P_{x_0}\{\tau\wedge \sigma(y)>k N_y\} \leq \Big(1 - \sup_{z \in [0,\,y]} \P_z\{\tau\wedge \sigma(y)\leq N_y\}\Big)^k\leq (1-\rho_y)^k.
\]
This shows that the $\mmp_{x_0}$-distribution of $\tau\wedge \sigma(y)$ has an exponential tail which particularly implies \eqref{ineq2}. Thus, formula \eqref{equa1} has been proved.

A combination of \eqref{eqn:stoppedMart} and \eqref{equa1} gives
\[
  f(x_0)= \E_{x_0} f(S_{\sigma(y)})\1_{\{\sigma(y)<\tau\}}+\me_{x_0} h(S_\tau)\1_{\{\sigma(y)\geq \tau\}}-\E_{x_0}\sum_{k=0}^{\tau\wedge \sigma(y)-1}g(S_k).
\]
Using Lemma \ref{lem:renewalRepresentation} and the estimate for $|f|$ we arrive at
\[  \limsup_{y \to \infty} \E_{x_0}f(S_{\sigma(y)})\1_{\{\sigma(y)<\tau\}}
\leq C(\limsup_{y \to \infty} \E_{x_0}S_{\sigma(y)}\1_{\{\sigma(y)<\tau\}}+1)=C(U(x_0)+1).\] Since $h$ is a bounded function on $(-\infty, 0]$ we infer $\lim_{y\to\infty}\me_{x_0} h(S_\tau)\1_{\{\sigma(y)\geq \tau\}}=\me_{x_0} h(S_\tau)=:C_1\in (-\infty,\infty)$ by the Lebesgue dominated convergence theorem. Invoking the L\'{e}vy monotone convergence theorem yields
\[  \lim_{y \to \infty} \E_{x_0} \sum_{k=0}^{\tau \wedge \sigma(y)-1}g(S_k)= \E_{x_0} \sum_{k=0}^{\tau-1} g(S_k)\]
By assumption, the right-hand side is infinite.
We conclude that necessarily
\[
  f(x_0) \leq C(U(x_0)+1)+C_1- \E_{x_0} \sum_{k=0}^{\tau-1}g(S_k)= - \infty,
\]
a contradiction which completes the proof of Lemma \ref{lem:firstStep}.
\end{proof}

\begin{lemma}\label{dri2}
Assume that condition \eqref{eqn:sumCondition} holds for some
$x>0$ and that the function $g$ is dRi on $\mr^+$. Then
\eqref{eqn:sumCondition} holds for each $x>0$ and
\begin{equation}\label{imp123}
\lim_{x\to\infty}\Big(\me_x \sum_{k=0}^{\tau-1}g(S_k)\Big)/x=0.
\end{equation}
\end{lemma}
\begin{proof}
We start by recalling that $\mu,\nu\in (0,\infty)$ according to
\eqref{moments}. By Lemma \ref{lem:renewIntegral},
\begin{equation}\label{repre}
\E_x\sum_{k=0}^{\tau- 1}g(S_k)=\int_{[0,\,x]}r(x-y)\dd U(y),\quad x>0,
\end{equation}
where $r(x) = \int_{[0,\,\infty)} g(x+z)\dd V(z)$ for $x\geq 0$. Thus, if
\begin{equation*}
\lim_{x\to\infty}r(x)=0,
\end{equation*}
then using the first part of \eqref{asymp} relation
\eqref{imp123} follows with the help of a simple (Stolz-Cesàro
like) argument.

By Lemma \ref{integr} (b), we infer $r(x)<\infty$ for each $x\geq
0$ which implies that \eqref{eqn:sumCondition} holds for each
$x>0$. The function $V$ is subadditive on $\mr$ (see, for
instance, formula (6.3) in \cite{Iksanov:2016}). Armed with this
we obtain, for each $x\geq 0$,
\begin{multline}\label{dri3}
r(x)\leq \int_{[\lfloor x\rfloor,\,\infty)}g(y){\rm d}V(y-x)\leq
\sum_{n\geq \lfloor x\rfloor+1}\sup_{n-1\leq y<n}g(y)
(V(n-x)-V(n-1-x))\\ \leq V(1)\sum_{n\geq \lfloor
x\rfloor+1}\sup_{n-1\leq y<n}g(y),
\end{multline}
where $z\mapsto \lfloor z\rfloor$ is the floor function. Since $g$
is dRi on $\mr^+$ and thereupon $$\overline{\sigma}(1)=\sum_{n\geq
1}\sup_{n-1\leq y<n}g(y)<\infty,$$ the right-hand side converges
to $0$ as $x\to\infty$.
\end{proof}

\begin{lemma} \label{lem:simple}
Assume that condition \eqref{eqn:sumCondition} holds for each $x>0$ and that $g$ is dRi on $\mr^+$. Then the function $f$ defined by
\[f(x):=\E_x h(S_\tau) - \E_x \sum_{k=0}^{\tau-1}g(S_k),\quad x\in\mr\]
is a solution to \eqref{eqn:problem}, and $\lim_{x \to \infty}(f(x)/x)= 0$.
\end{lemma}
\begin{proof}
Let us check that $f$ is a solution to \eqref{eqn:problem} which exhibits at most linear growth. Using the fact that, under $\mmp_x$, $x\in\mr$, $(S_k-S_1)_{k\in\mn}$ has the same distribution as $(S_n-x)_{n\in\mn_0}$ and is independent of $S_1$ and that, by definition,
\begin{equation}\label{inter1}
\tau= 0\quad \P_x-\text{a.s.\ ~ for}~ x\leq 0,
\end{equation}
we obtain
\[
  f(x) = \E f(x + \xi)-g(x),\quad x>0.
\]
Also, $$f(x)=h(x),\quad x\leq 0$$ by another appeal to \eqref{inter1} (in particular, $\E_x \sum_{k=0}^{\tau-1}g(S_k)=0$ for $x<0$).

Next, we note that $\lim_{x\to\infty}(f(x)/x)=0$ is a consequence of Lemma \ref{dri2} and boundedness of $h$.

Finally, we show that the function $f$ is right-continuous. By
assumption, $h$ is a right-continuous bounded function. Hence, the
function $x\mapsto \E_x h(S_\tau)$ is right-continuous by the
Lebesgue dominated convergence theorem. To prove right-continuity
of $x\mapsto \E_x \sum_{k=0}^{\tau-1}g(S_k)$ on $(0,\infty)$ we
are going to use representation \eqref{repre}. For $x,y\geq 0$ and $z\in [0,1]$,
\begin{multline*}
g(x+z+y)\leq \sum_{n\geq 1}\sup_{n-1+z\leq y<n+z}g(x+y)\1_{[n-1,
n)}(y)\\ \leq \sum_{n\geq 1}\sup_{n-1\leq y<n+1} g(x+y)\1_{[n-1,
n)}(y)=:L_x(y).
\end{multline*}
Also,
\begin{multline*}
\int_{[0,\,\infty)}L_x(y){\rm d}V(y)=\sum_{n\geq 1}\sup_{n-1\leq
y<n+1} g(x+y)(V(n-)-V((n-1)-))\\\leq V(1)\sum_{n\geq
1}\sup_{n-1\leq y<n+1} g(x+y)<\infty,
\end{multline*}
where the finiteness is secured by the fact that $g$ is dRi and
the penultimate inequality is justified by subadditivity of $V$ on
$\mr$. Hence,
$$\lim_{z\to 0+}r(x+z)=\int_{[0,\,\infty)}\lim_{z\to 0+}
g(x+z+y){\rm d}V(y)=\int_{[0,\,\infty)}g(x+y){\rm d}V(y)=r(x)$$ by
right-continuity of $g$ and the Lebesgue dominated convergence
theorem. According to the proof of Lemma \ref{dri2}, $\lim_{x\to\infty}r(x)=0$, whence $$r(x+z-y)\1_{[0,\,x+z]}(y) \leq c\1_{[0,\,x+1]}(y)$$
for $x\geq 0$, $z\in [0,1]$, $y\in [0,x+z]$ and a constant $c>0$. Thus, we infer
\begin{multline*}
\lim_{z\to 0+}\int_{[0,\,\infty)}r(x+z-y)\1_{[0,\,x+z]}(y){\rm
d}U(y)=\int_{[0,\,\infty)}\lim_{z\to 0+}
r(x+z-y)\1_{[0,\,x+z]}(y){\rm d}U(y)\\=\int_{[0,\,x]}r(x-y){\rm
d}U(y)
\end{multline*}
by another appeal to the Lebesgue dominated convergence theorem.
Thus, right-continuity on $(0,\infty)$ has been proved. By a
similar reasoning, one can also check that $\lim_{x\to 0+}\E_x
\sum_{k=0}^{\tau-1}g(S_k)=0$.
\end{proof}
\begin{remark}\label{arith}
Assume that the distribution of $\xi$ is $d$-arithmetic for $d>0$
and the function $g$ is not dRi. Then it can be checked (details are simple, hence omitted) that if
\eqref{eqn:sumCondition} holds for some $x>0$, then $\sum_{n\geq
0}g(x+nd)<\infty$ and thereupon
$\lim_{n\to\infty}(H(x+nd)/(nd))=0$, where $H(y):=\me_y \sum_{k=0}^{\tau-1}g(S_k)$ for $y>0$. However,
this does not seem to imply
${\lim\sup}_{x\to\infty}(H(x)/x)<\infty$ which is needed for
proving that $f$ defined in Lemma \ref{lem:simple} satisfies
${\lim\sup}_{x\to\infty}(|f(x)|/x)<\infty$, let alone
$\lim_{x\to\infty}(H(x)/x)=0$. On the other hand, the assumption
that $g$ is dRi comfortably ensures the latter.
\end{remark}

We now turn to the proof of the main result of the section.
\begin{proof}[Proof of Theorem \ref{prop:gen}]
In view of Lemma \ref{lem:firstStep} it remains to consider the
case when condition \eqref{eqn:sumCondition} holds for some $x>0$
and $g$ is dRi on $\mr^+$. Then, by Lemma \ref{dri2},
\eqref{eqn:sumCondition} holds for each $x>0$. Hence, Lemma
\ref{lem:simple} applies and ensures that $x\mapsto \E_x h(S_\tau)-\E_x \sum_{k=0}^{\tau-1}g(S_k)$, $x\in\mr$ is a solution to
\eqref{eqn:problem} of sublinear growth.

Let $f$ be any solution \eqref{eqn:problem} for which
\eqref{eqn:asy} holds. Lemma \ref{lem:simple} in combination with
the linearity of \eqref{eqn:problem} enables us to conclude that
the function $\hat f$ defined by
\[
  \hat{f}(x):=f(x) - \E_x h(S_\tau) + \E_x \sum_{k=0}^{\tau-1}g(S_k),\quad x\in\mr
\]
satisfies
\[
  \begin{cases}
    \hat f(x) = \E \hat f(x+\xi)\1_{\{\xi>-x\}}, &\text{ if } x> 0,\\
    \hat f(x) = 0, &\text{ if } x \leq 0\\
    \limsup_{x \to \infty}(|\hat f(x)|/x) < \infty.
  \end{cases}
\]
In other words, $\hat{f}$ is a harmonic function of at most linear
growth for the random walk $S$ killed upon entering $(-\infty,
0]$. Therefore, the proof is completed by an application of Lemma
\ref{lem:harmonicKilled}.
\end{proof}

As a consequence of Theorem \ref{prop:gen}, we conclude that subharmonic functions of at most linear growth for the killed random walk exhibit exactly a linear growth rate (at least along the closure of the group generated by the support of the distribution of $\xi$).
\begin{corollary}\label{cor:linearGrowth}
Assume that the distribution of $\xi$ is $d$-arithmetic for $d\geq 0$ and that the function $g$ is dRi on $\mr^+$.
Let $f$ be a solution to \eqref{eqn:problem} satisfying \eqref{eqn:asy} and $\kappa(\cdot)$
the corresponding $d$-periodic function from \eqref{impo1}.  Then
\[
\begin{cases}
\lim_{n \to \infty} (f(x + nd)/nd) = \kappa(x)/\mu \quad \text{for all}~ x\in [0,d), & \text{ if } d > 0\\
\lim_{x \to \infty} (f(x)/x) = \kappa/\mu, & \text{if}~ d=0.
  \end{cases}
\]
Furthermore, if $\kappa(x)=\kappa$ for all $x\in\mr$ in the case $d>0$, then $\lim_{x \to \infty} (f(x)/x) = \kappa/\mu$.
\end{corollary}

This result follows from Theorem \ref{prop:gen}, Lemma \ref{lem:simple} and \eqref{asymp}.

\section{Proofs related to tail behavior}\label{proofs1}
Recall that the random variable $Z$ is the a.s.\ limit of the
derivative martingale $(Z_n, \F_n)_{n\in\mn_0}$.
In this section we prove Theorems \ref{expa0} and \ref{expa} by investigating the asymptotic behavior of the Laplace transform of $Z$ near zero and using Tauberian theorems given in the Appendix.

\subsection{Decomposition of \texorpdfstring{$Z$}{Z}}\label{decomp}
Let $\eta_1$, $\eta_2,\ldots$ be independent copies of a random
variable $\eta$ which are independent of $\cZ=\sum_{j=1}^N
\delta_{X_j}$. The mapping which maps the distribution of $\eta$
to the distribution of $\sum_{i=1}^N e^{-X_i}\eta_i$ is an
instance of {\it smoothing transform}. The distribution of $\eta$
is a {\it fixed point} of this smoothing transform if
$$\eta\overset{{\rm d}}{=}\sum_{i=1}^N e^{-X_i}\eta_i,$$ where $\overset{{\rm d}}{=}$ denotes
equality of distributions. Recent advances concerning fixed points
of general smoothing transforms can be found in
\cite{Alsmeyer+Mallein:2019+, Alsmeyer+Meiners:2013,
Iksanov+Meiners:2015, Meiners+Mentemeier:2017}, the list is far
from being complete.

Denote by $\phi$ the Laplace transform of $Z$, that is,
\[
\phi(s) = \E e^{-sZ},\quad s\geq 0.\] Below we provide
an a.s.\ decomposition of $Z$ over the individuals of any fixed
generation. The distributional version of formula \eqref{fixpoint}
in the case $k=1$ shows that the distribution of $Z$ is a fixed
point of the particular smoothing transform. This fact
reformulated in terms of $\phi$ reads
\begin{equation}\label{eqn:fixedpointDistributional}
\phi(s) = \E\prod_{i=1}^N \phi(s e^{-X_i}), \quad s \geq 0.
\end{equation}

As a preparation, we recall from Lemma 3.1 in \cite{Shi:2015}
that, under \eqref{eq:m(1)=1} and \eqref{eqn:boundarycase}, we
have
\begin{equation}\label{div}
\lim_{n \to \infty} \inf_{|u|=n} S(u) = \infty\quad \text{ a.s.},
\end{equation}
that is, the minimal position of the $n$th generation individuals
diverges to $\infty$ as $n\to\infty$. Here, the infimum is defined
to be $+\infty$ if the population dies out by the nth generation.
Further, for $u,v\in\I$ we write $v>u$ if $u$ is an ancestor of
$v$, that is, $u=u_1\ldots u_k$ and $v=u_1\ldots u_k\ldots u_n$
for some $k\in\mn_0$ and integer $n>k$. Given $u\in \I$, set
$$Z_n(u):=\sum_{|v|=n+|u|,\, v>u}e^{-(S(v)-S(u))}(S(v)-S(u)),\quad n\in\mn,$$ so that
$(Z_n(u))_{n\in\mn}$ is a version of $(Z_n)_{n\in\mn}$. Then
$$Z(u):=\lim_{n\to\infty} Z_n(u)$$ is the a.s.\ limit of the derivative martingale defined on the
subtree of $\I$ rooted at $u$. For fixed $k\in\mn$, the random
variables $(Z(u))_{|u|=k}$ are independent copies of $Z$ which are
also independent of $(S(u))_{|u|=k}$.
\begin{lemma} \label{lem:fixpoint} Assume that Condition $\mathcal{S}$ holds.
Then, for each $k\in\mn$,
\begin{equation}\label{fixpoint}
Z=\sum_{|u|=k}e^{-S(u)}Z(u)\quad\text{{\rm a.s.}}
\end{equation}
\end{lemma}
\begin{remark}
In the situation where $N<\infty$ a.s.\ this fact was proved in
Theorem 5.1 of \cite{Biggins+Kyprianou:2004}. However, we work
under weaker assumptions, in particular,  the case
$\mmp\{N=\infty\}>0$ is not excluded in the present work. Since we
did not find an appropriate reference in the literature, we give a complete
proof.
\end{remark}
\begin{proof}
Let $(\tau^\ast_k)_{k\in\mn_0}$ be the sequence of strict
descending ladder epochs, that is, $\tau^\ast_0:=0$,
$\tau^\ast_1:=\inf\{j\in\mn: S_j<0\}$ and
$\tau^\ast_k:=\inf\{j>\tau^\ast_{k-1}: S_j<S_{\tau^\ast_{k-1}}\}$
for $k\geq 2$. Put $$R(x):=\sum_{n\geq 0}\mmp\{-S_{\tau^\ast_n}\leq
x\},\quad x\in\mr,$$ that is, $R$ is the renewal function for the
standard random walk formed by strict descending ladder heights.
Note that $R(x)=0$ for $x<0$. For fixed $\alpha\geq 0$, put
$$D_n^{(\alpha)}=\sum_{|u|=n} e^{-S(u)}R(S(u)+\alpha)\1_{\{S(u_1)\ge -\alpha, S(u_1u_2)\ge-\alpha,\ldots, S(u_1\ldots u_n)\ge -\alpha\}},\quad n\in\mn_0$$ and let $A_\alpha:=\{S(u) \ge
-\alpha\quad\text{for all ever born individuals}~u\}$ denote the
event of nonextinction of the branching random walk killed below
$-\alpha$.  According to Lemma A.1 in \cite{Aidekon:2013}, the
sequence $(D_n^{(\alpha)},\mathcal{F}_n)_{n\in\mn_0}$ forms a
nonnegative martingale called {\it truncated martingale}.
Furthermore, by Proposition A.3 in \cite{Aidekon:2013},
$D_n^{(\alpha)}$ converges a.s.\ and in $L^1$ as $n\to\infty$ to a
random variable that we denote by $D^{(\alpha)}$, and
$D^{(\alpha)}>0$ a.s.\ on $A_\alpha$.

By Lemma \ref{don} (a,b),
\begin{equation}\label{rx}
\lim_{x\to\infty}x^{-1}R(x)=(-\me S_{\tau_1^\ast})^{-1}=:{\tt
m}^{-1}>0.
\end{equation}
This together with \eqref{div} enables us to conclude that, a.s.\ on $A_\alpha$,
\begin{equation}\label{x}
D^{(\alpha)} = \lim_{n \to \infty} \sum_{|u|=n}e^{-S(u)}R(S(u) +
\alpha)= \lim_{n \to \infty} \sum_{|u|=n} {\tt
m}^{-1}e^{-S(u)}S(u) = {\tt m}^{-1} Z
\end{equation}
(we note in passing that these random variables are
not equal a.s.\ because $\me D^{(\alpha)}<\infty$, whereas $\me
Z=\infty$). We extend the definition of the truncated martingale
to the
subtrees rooted at $u\in \I$ as follows
$$D_n^{(\alpha)}(u):= \sum_{|v|=n+|u|,\, v>u} e^{-S(v)}R(S(v)+\alpha)
\1_{\{S(uv_1) \ge -\alpha, S(uv_1v_2)\ge-\alpha,\ldots, S(uv_1\ldots v_n)\ge -\alpha\}}, \quad n\in\mn.$$ Fix $u\in\I$. The sequence $(e^{S(u)}
D_n^{(\alpha)}(u))_{n\in\mn}$ has the same distribution as $(D_{n,\ast}^{(S(u)+\alpha)})_{n\in\mn}$, where, while for
$\beta\geq 0$, $(D_{n,\ast}^{(\beta)})_{n\in\mn}$ is a
distributional copy of $(D_n^{(\beta)})_{n\in\mn}$ which is
independent of $S(u)$; for $\beta<0$, $D_{n,\ast}^{(\beta)}=0$ for
each $n\in\mn$. From this we conclude that $D_n^{(\alpha)}(u)$
converges a.s.\ and in $L^1$, as $n\to\infty$,
to a random variable $D^{(\alpha)}(u)$, say which satisfies
\begin{equation}\label{y}
D^{(\alpha)}(u) = {\tt m}^{-1}e^{-S(u)}Z(u) \qquad \mbox{a.s. on }
A_\alpha
\end{equation}
and
\begin{equation}\label{eq: ee}
\E\big(D^{(\alpha)}(u)|{\mathcal F}_{|u|}\big) =
\E\big(D_n^{(\alpha)}(u)|{\mathcal F}_{|u|}\big) =
e^{-S(u)}R(S(u)+\alpha)\qquad \mbox{a.s.}
\end{equation}
Decomposing $D_n^{(\alpha)}$ over the $k$th generation yields
$$D_n^{(\alpha)} = \sum_{|u|=k} D_{n-k}^{(\alpha)}(u),\quad n>k\quad\text{a.s.}$$
By Fatou's lemma, for $k\in\mn$,
\[
  D^{(\alpha)} \geq \sum_{|u|=k} D^{(\alpha)}(u) \geq 0 \quad \text{a.s.}
\]
Also, for $k\in\mn$,
\[
  \E D^{(\alpha)}=R(\alpha)=\E R(S_k+\alpha) = \E \sum_{|u|=k}e^{-S(u)}R(S(u)+\alpha)=\E
  \sum_{|u|=k}D^{(\alpha)}(u),
\]
where the first and the last equalities follow from \eqref{eq: ee}, the second equality expresses the known fact that $R$ is a
harmonic function of the random walk $S$ killed upon entering
$(-\infty, 0)$ (see Lemma 1 in \cite{Tanaka:1989}), and the third
equality is a consequence of Lemma \ref{many}. The last two
centered formulae together ensure that, for $k\in\mn$,
\[
  D^{(\alpha)}= \sum_{|u|=k} D^{(\alpha)}(u)\quad \text{a.s.}
\]
Using \eqref{x} and \eqref{y} yields, for each $\alpha\geq 0$ and
$k\in\mn$,
\[
  Z = \sum_{|u|=k} e^{-S(u)}Z(u)\qquad \text{ a.s. on } A_\alpha,
\]
hence just a.s.\ because $(A_\alpha)_{\alpha\geq 0}$ is a
nondecreasing family of events with
$\lim_{\alpha\to\infty}\mmp(A_\alpha)=1$.
\end{proof}

\subsection{Asymptotic behavior of the Laplace transform}\label{LT}

Recall that $\phi$ denotes the Laplace transform of $Z$ and put
\[
D(x) = e^x (1-\phi(e^{-x})),\quad x\in\mr.
\]
\begin{remark}\label{55}
Assume that Condition $\mathcal{S}$ holds. Then, according to Lemma
5.1 in \cite{Alsmeyer+Mallein:2019+},
\begin{equation}\label{eqn:boundforphi}
\sup_{x>0}\frac{D(x)}{1+x}< \infty.
\end{equation}
Theorems \ref{impo} and \ref{expaLT} given below in this section can be thought of as a strengthening of
\eqref{eqn:boundforphi}.
\end{remark}

Following Durrett and Liggett \cite{Durrett+Liggett:1983} and many
their successors we put, for $x\in\mr$,
\begin{align*}
G(x) &= \E \sum_{i=1}^N e^{-X_i}D(x + X_i)-D(x)=\me D(x+\xi)-D(x)\\
&= e^x \E\Big( \prod_{i=1}^N \phi(e^{-x-X_i}) - 1 + \sum_{i=1}^N
\big(1 - \phi(e^{-x-X_i})\big)\Big),
\end{align*}
where $\xi$ is a random variable with distribution defined
in \eqref{xidef}. To obtain the second equality we have used
\eqref{eqn:fixedpointDistributional}. For later needs, we note the following.
\begin{lemma}\label{durligg}

\noindent (a) $G(x)\geq 0$ for $x\in\mr$;

\noindent (b) the function $x\mapsto e^{-x} G(x)$ is nonincreasing
on $\mr$.
\end{lemma}
These two properties were given in Lemma 2.4 of
\cite{Durrett+Liggett:1983} under the assumption that $N$ is
deterministic. However, the proof of the cited result extends
verbatim to the more general situation treated here.

From the definition of $G$ and formula
\eqref{eqn:boundforphi} it follows that $D$ satisfies
\begin{equation}\label{eqn:harmonicD}
\begin{cases}
D(x) = \E D(x+ \xi)-G(x),\quad x\in\mr;&\\
\sup_{x \in \mr} \frac{D(x)}{1+|x|} < \infty.
\end{cases}
\end{equation}
In particular, $D$ is a nonnegative subharmonic function of at
most linear growth for the random walk $S$. Therefore, invoking Theorem \ref{prop:gen} we can give an
alternative formula for $D$. Below we use the notation introduced
in Section \ref{nota}.
\begin{theorem}\label{impo}
Assume that Condition $\mathcal{S}$ holds. Then, for each $x>0$,
\begin{equation}\label{repr}
D(x) = \mu U(x) + \E_x D(S_\tau)- \E_x\sum_{k=0}^{\tau-1} G(S_k)
\end{equation}
and
\begin{equation}\label{limit1} D(x)~\sim~
x,\quad x\to\infty.
\end{equation}
Also, if the distribution of $\xi$ is nonarithmetic, then the limit $\lim_{x\to\infty} \me_x D(S_\tau)$ exists and is finite. If the distribution of $\xi$ is $d$-arithmetic for $d>0$, then the limit does not exist but
\begin{equation}\label{period}
\lim_{x\to\infty} (\me_x D(S_\tau)-c_{11}(x))=0
\end{equation}
for a bounded $d$-periodic function $c_{11}(\cdot)$ which is not a constant.
\end{theorem}

Theorem \ref{expa0} is an immediate consequence of
\eqref{limit1} and Corollary 8.1.7 in
\cite{Bingham+Goldie+Teugels:1989} which states that relations
\eqref{limit1} and \eqref{limit2} are equivalent.
\begin{theorem}\label{expaLT}
Assume that Condition $\mathcal{S}_{\rm na}$ holds. Then Condition $\mathcal{S}^\ast$ ensures
\begin{equation}\label{limit123}
D(x)=\mu U(x)+c_1+o(1)=x+c_2+o(1),\quad x\to\infty,
\end{equation}
where $c_2=c_1+(2\mu)^{-1}\me S_\tau^2=c+1-\gamma$, $\gamma$ is the Euler-Mascheroni constant, and $c$ is the same as in Theorem \ref{limit3}. Conversely, the second equality in \eqref{limit123} entails Condition $\mathcal{S}^\ast$.
\end{theorem}

At this point it is convenient to prove Theorem \ref{expa}.
\begin{proof}[Proof of Theorem \ref{expa}]
Assume first that Condition $\mathcal{S}^\ast$ holds. While formula \eqref{limit3} of Theorem \ref{expa} follows from the second equality in \eqref{limit123} and the equivalence (I) $\Rightarrow$ (III) of Lemma \ref{link}, formula \eqref{Z} is a consequence of the fact that \eqref{limit123} entails $\lim_{x\to\infty}(D(x+y)-D(x))=y$ for each $y\in\mr$ and the implication  (ii)$\Rightarrow$ (i) of Lemma \ref{link1}. Assume now that representation \eqref{limit3} holds true. By the implication (III) $\Rightarrow$ (I) of Lemma \ref{link}, the second equality in \eqref{limit123} holds. With this at hand, the necessity of Condition $\mathcal{S}^\ast$ follows from Theorem \ref{expaLT}.
\end{proof}
\begin{remark}
Assume that Conditions $\mathcal{S}$ and $\mathcal{S}^\ast$ hold and that the distribution of $\xi$ is $d$-arithmetic for $d>0$. Although limit relation \eqref{limit123} cannot hold,
there exists a bounded $d$-periodic function $c_1(\cdot)$ which is not a constant such that
\begin{equation}\label{limit124}
D(x) = \mu U(x) + c_1(x) + o(1), \quad x \to \infty.
\end{equation}
Details can be found in Remark \ref{rem:finite}.
\end{remark}

\subsection{Proof of Theorem \ref{impo}}

By Lemma \ref{don} (a), the assumption $\me
\xi_-^2=\me \sum_{i=1}^N e^{-X_i}(X_i)_-^2<\infty$ which is one
half of \eqref{eqn:variance} ensures that $\mu=-\me S_{\tau_1}$ is
finite.

In view of \eqref{eqn:harmonicD}, the function $D$ satisfies \eqref{eqn:asy} and is a continuous solution to \eqref{eqn:problem}
with $h(x)=D(x)$ for $x\leq 0$ and $g = G$. Note that $D$ is
bounded on $(-\infty, 0]$ in view of
$$D(x)=e^x(1-\phi(e^{-x}))\leq e^x\leq 1,\quad x\leq 0,$$ and that $G$ is continuous. Let us
show that $G$ is dRi on $\mr^+$. Let $h_0=d$ if the distribution
of $\xi$ is $d$-arithmetic for $d>0$ and $h_0>0$ be arbitrary if
the distribution of $\xi$ is nonarithmetic. By Theorem \ref{prop:gen}, $\me_x\sum_{k=0}^{\tau-1}G(S_k)<\infty$ for each
$x>0$, hence for
$x=h_0$. Then using Lemma \ref{lem:renewIntegral} with $p=G$ to justify the
first inequality we infer
\begin{equation*}
\infty>r(h_0)=\int_{[h_0,\,\infty)}G(y){\rm d}V (y-h_0)\geq
\sum_{n\geq 1}\inf_{(n-1)h_0\leq
y<nh_0}\,G(y)(V((n-1)h_0)-V((n-2)h_0)).
\end{equation*}
By the Blackwell theorem (see, for instance, formulae (6.8) and
(6.9) in \cite{Iksanov:2016}),
\[\lim_{n\to\infty}(V((n-1)h_0)-V((n-2)h_0))=h_0/\nu \in
(0,\infty).\]
Therefore,
$$\underline{\sigma}(h_0)=h_0\sum_{n\geq
1}\inf_{(n-1)h_0\leq y<nh_0}\,G(y)<\infty.$$
By Lemma \ref{dri},
this together with the fact that the function $x\mapsto
e^{-x}G(x)$ is nonincreasing (see Lemma \ref{durligg}) enables us
to conclude that $g$ is dRi on $\mr^+$.

By Theorem \ref{prop:gen}, there exists a $d$-periodic
function $\kappa(\cdot)$ such that, for all $x>0$,
\[
  D(x) = \kappa (x)U(x) + \E_x D(S_\tau)- \E_x \sum_{k=0}^{\tau-1} G(S_k)=:\kappa (x)U(x)+r(x).
\]
To complete the proof of \eqref{repr} we have to show that
$\kappa(x)=\mu$ for all $x>0$. Relation \eqref{limit1} will then
follow by Lemma \ref{lem:simple} and Corollary \ref{cor:linearGrowth}.

The subsequent argument is close to the discussion in
\cite{Biggins+Kyprianou:1997}, particularly to Theorem 3.1 and
Lemma 5.1 therein.
Using Lemma \ref{lem:fixpoint} yields, for $\lambda>0$,
\begin{align*}
\E(e^{-\lambda Z}| \mathcal{F}_n) &= \E\Big(\exp \Big(-\sum_{|u|=n} \lambda e^{-S(u)} Z(u) \Big) \Big| \mathcal{F}_n \Big)\\
  &= \prod_{|u|=n} \phi(\lambda e^{-S(u)}) = \prod_{|u|=n} ( 1 - \lambda D(S(u) - \log \lambda)e^{-S(u)}).
\end{align*}
Since $\big(\E(e^{-\lambda Z}| \mathcal{F}_n), \F_n\big)_{n\in\mn_0}$ is a right closable martingale,
we have $\lim_{n\to\infty}\E(e^{-\lambda Z}| \mathcal{F}_n)=e^{- \lambda Z}$ a.s.\ and thereupon
\[
  Z =(1/\lambda) \lim_{n \to \infty}\sum_{|u|=n} - \log \big( 1 - \lambda D(S(u) - \log \lambda) e^{-S(u)} \big)\quad\text{a.s.}\]
On the other hand, for all $\lambda>0$ and $u$ with $|u|=n$,
\[
  D(S(u) - \log \lambda) = \kappa(- \log \lambda)U(S(u) - \log \lambda) + r(S(u) - \log \lambda).
\]
We have used the equality $\kappa(S(u)-\log \lambda)=\kappa(- \log \lambda)$ which is trivial if the distribution of $\xi$ is nonarithmetic, for $\kappa(\cdot)$ is then a constant. If the distribution of $\xi$ is $d$-arithmetic for $d>0$, the equality is secured by $S(u) \in d\Z$ a.s. Further, we conclude that, for all $\lambda>0$,
\begin{align*}
\lim_{n \to \infty} \sum_{|u|=n} - \log \left( 1 - \lambda D(S(u) - \log \lambda) e^{-S(u)} \right) &= \lim_{n \to \infty} \sum_{|u|=n} (\lambda/\mu)\kappa(-\log \lambda) S(u) e^{-S(u)}\\
& = (\lambda/\mu) \kappa(-\log \lambda) Z \quad \text{a.s.}
\end{align*}
having utilized $e^{-x}D(x)=1-\phi(e^{-x})\to 0$ as $x\to\infty$,
$\lim_{n \to \infty} \inf_{|u|=n} S(u) = \infty$ a.s., the last
centered formula, Lemma \ref{don}(b) and
$\lim_{x\to\infty}(r(x)/x)=0$ (see Lemma \ref{lem:simple}) for the
first equality. Since $\mmp\{Z>0\}>0$, we infer $\kappa(-\log
\lambda) =\mu$ for all $\lambda >0$.

It remains to investigate the existence of the limit $\lim_{x\to\infty}\me_x D(S_\tau)$.
For $x>0$, put $\tau(-x):=\inf\{k\in\mn: S_k\leq -x\}$ and
$\nu(x):=\inf\{k\in\mn: -S_{\tau_k}\geq x\}$, so that $\nu(x)$ is
the first passage time into $[x,\infty)$ of
$(-S_{\tau_k})_{k\in\mn_0}$. Then
\begin{equation}\label{imp1}
\me_x D(S_\tau)=\me D(x+S_{\tau(-x)}),\quad x>0.
\end{equation}
Since the first passage into $(-\infty, -x]$ of
$(S_k)_{k\in\mn_0}$ can only occur at a weakly descending ladder
epoch we infer
\begin{equation}\label{imp2}
S_{\tau(-x)}=S_{\tau_{\nu(x)}}\quad \text{a.s.}
\end{equation}
Hence, $$\me_x D(S_\tau)=\me
D(-(-S_{\tau_{\nu(x)}}-x))=\int_{[0,\,x)}m(x-y){\rm d}U(y),$$
where $m(x):=\me D(-(-S_{\tau_1}-x))\1_{\{-S_{\tau_1}\geq x\}}$
for $x\geq 0$. Boundedness and continuity of $D$ on $(-\infty, 0]$
implies that $m$ is locally Riemann integrable on $(-\infty, 0]$.
Since $m(x)\leq \mmp\{-S_{\tau_1}\geq x\}$ for $x\geq 0$ and
$x\mapsto \mmp\{-S_{\tau_1}\geq x\}$ is dRi on $\mr^+$ as a
nonincreasing and Lebesgue integrable function (note that
$\int_0^\infty \mmp\{-S_{\tau_1}\geq x\}{\rm d}x=\mu<\infty$) we
conclude that $m$ is dRi on $\mr^+$. It is known (see Lemma
\ref{don} (c)) that the distribution of $S_{\tau_1}$ is
$d$-arithmetic because so is the distribution of $\xi$. Thus,
invoking the key renewal theorem yields
$$\lim_{x\to\infty}\me_x D(S_\tau)=\mu^{-1}\int_0^\infty m(y){\rm d}y=\mu^{-1}\int_0^\infty D(-y)\mmp\{-S_{\tau_1}\geq y\}{\rm d}y<\infty$$
in the nonarithmetic case $d=0$ (see, for instance, Proposition
6.2.3 in \cite{Iksanov:2016}), whereas, for each $x\in [0,d)$,
$$\lim_{n\to\infty}\me_{x+nd}D(S_\tau)=d\mu^{-1}\sum_{k\geq
0}m(x+kd)=:\tilde m(x)<\infty$$ in the arithmetic case $d>0$ (see,
for instance, Proposition 6.2.6 in \cite{Iksanov:2016}). Setting
$c_{11}(x):=\tilde m(d\{x/d\})$ for $x\geq 0$, where $\{y\}$ is
the fractional part of $y$, we observe that  the last limit
relation is equivalent to \eqref{period}. It is clear that
$c_{11}(\cdot)$ is a bounded $d$-periodic function. To see that it
is not a constant (which implies that the limit
$\lim_{x\to\infty}\me_x D(S_\tau)$ does not exist) one may use,
for instance, the fact $x\mapsto e^{-x}m(x)$ is a nonincreasing
function which follows from Lemma \ref{durligg} (b). The proof of Theorem \ref{impo} is complete.

\subsection{Proof of Theorem \ref{expaLT}}\label{proof3}

For the proof of Theorem \ref{expaLT} we need some more preparations.
\begin{lemma}\label{finite}
Assume that Condition $\mathcal{S}_{\rm na}$ holds. Then the limit $\lim_{x\to\infty}\E_x \sum_{k=0}^{\tau-1}G(S_k)$
exists and is finite if, and only if, $\int_0^\infty yG(y){\rm d}y<\infty$.
\end{lemma}
\begin{proof}
The definitions of the strict ascending ladder epochs $(\sigma_n)_{n\in\mn_0}$ and the renewal functions $U$ and $V$ are given in Section \ref{nota}.
We first recall that by Lemma \ref{lem:renewIntegral}
\begin{equation}\label{imp3}
\me_x\sum_{j=0}^{\tau-1}G(S_j)=\int_{[0,\,x]}r(x-y){\rm
d}U(y),\quad x>0,
\end{equation}
where $r(z)=\int_{[0,\,\infty)}G(z+y){\rm d}V(y)=\me\sum_{k\geq
0}G(z+S_{\sigma_k})$ for $z\geq 0$.

Next, we are going to prove that $r$ is Lebesgue integrable on
$\mr^+$ if, and only if, $\int_0^\infty yG(y){\rm d}y<\infty$. By
a multiple use of Fubini's theorem,
\begin{multline*}
\int_0^\infty r(y){\rm d}y=\int_0^\infty \me\sum_{k\geq
0}G(y+S_{\sigma_k}){\rm d}y=\me\sum_{k\geq
0}\int_{S_{\sigma_k}}^\infty G(y){\rm d}y=\sum_{k\geq
0}\me\int_0^\infty G(y)\1_{\{S_{\sigma_k}\leq y\}}{\rm
d}y\\=\int_0^\infty G(y)V(y){\rm d}y,
\end{multline*}
where all the integrals are either convergent or divergent
simultaneously. While by Lemma \ref{don} (a), the condition $\me
\xi_+^2<\infty$ guarantees $\me S_{\sigma_1}<\infty$, we obtain with the help of
Lemma \ref{don}(b) that $\lim_{y\to\infty}y^{-1}V(y)=(\me
S_{\sigma_1})^{-1}\in (0,\infty)$. Hence, the last integral
converges if, and only if, so does the integral $\int_0^\infty
yG(y){\rm d}y$.

Assume now that $\int_0^\infty yG(y){\rm d}y<\infty$, hence
$\int_0^\infty r(y){\rm d}y<\infty$. According to Lemma
\ref{durligg}(b), $x\mapsto e^{-x}G(x)$ is a nonincreasing function on $\mr$. Hence,
so is $x\mapsto e^{-x}r(x)$ which implies that $r$ is a dRi
function on $\mr^+$, see Lemma \ref{dri}. Recalling that the
distribution of $S_{\tau_1}$ is nonarithmetic (because so is the
distribution of $\xi$), that $\mu=-\me S_{\tau_1}<\infty$ and invoking the key renewal theorem we
infer
$$\me_x\sum_{j=0}^{\tau-1}G(S_j)=\int_{[0,\,x]}r(x-y){\rm
d}U(y)~\to~\mu^{-1}\int_0^\infty r(y){\rm d}y\in (0,\infty),\quad x\to\infty.$$

Finally, assume that $\int_0^\infty yG(y){\rm d}y=\infty$, so that
$\int_0^\infty r(y){\rm d}y=\infty$. We already know that the
function $x\mapsto e^{-x}r(x)$ is nonincreasing on $\mr^+$. Hence,
$r$ is locally bounded and a.e.\ continuous on $\mr^+$. This
implies that, for each $b>0$, the function $x\mapsto
r(x)\1_{[0,\,b]}(x)$ is dRi. Now an application of the key renewal
theorem yields
$${\lim\inf}_{x\to\infty}\int_{[0,\,x]}r(x-y){\rm d}U(y)\geq \lim_{x\to\infty}\int_{(x-b,
x]}r(x-y){\rm d}U(y)=\mu^{-1}\int_0^b r(y){\rm d}y.$$ Letting
$b\to\infty$ we conclude that
$\lim_{x\to\infty}\int_{[0,\,x]}r(x-y){\rm d}U(y)=\infty$.
\end{proof}
\begin{remark}\label{rem:finite}
The constant in \eqref{limit123} is given by $c_1=\lim_{x\to\infty}(\me_x D(S_\tau)-\me_x\sum_{k=0}^{\tau-1}G(S_k))$. Assume that Condition $\mathcal{S}$ holds and that the distribution of $\xi$ is $d$-arithmetic for $d>0$. Then, according to Theorem \ref{impo}, the limit $\lim_{x\to\infty}\me_x D(S_\tau)$ does not exist which implies that \eqref{limit123} cannot hold. Under the additional assumption $\int_0^\infty yG(y){\rm d}y<\infty$ a minor modification of the proof of Lemma \ref{finite}, along the lines of the argument leading to \eqref{period}, yields $$\lim_{x\to\infty}\Big(\me_x \sum_{k=0}^{\tau-1}G(S_k)-c_{12}(x)\Big)=0$$ for a bounded $d$-periodic function $c_{12}(\cdot)$ which is not a constant. This in combination with \eqref{period} justifies \eqref{limit124} with $c_1(\cdot):=c_{11}(\cdot)-c_{12}(\cdot)$. Here, $c_{11}(\cdot)$ is the same as in \eqref{period}.
\end{remark}

\begin{lemma}\label{lem:d1}
Assume that Condition $\mathcal S$ holds. Then \eqref{eq:2.4} and  \eqref{eq:2.5} are sufficient for $\int_0^{\infty} yG(y){\rm d}y <\infty$.
\end{lemma}
Before giving a proof of Lemma \ref{lem:d1} we need an auxiliary  result.
\begin{lemma}\label{lem:lambert}
Let $a,b$ and $\varepsilon$ be real numbers satisfying $a>0$, $b\ge 0$, $c:=\log a-b/a\geq 0$ and $\varepsilon\in (0,1/e)$. The equation
\begin{equation}\label{eq:nd1}
ay-b=\varepsilon e^y
\end{equation}
has two solutions
\begin{equation}\label{eq:nd2}
y_1 = d + b/a\quad\text{and}\quad  y_2=-\log \varepsilon + \log a + \log \big( -\log \varepsilon + \log a - b/a \big) + o(1)
\end{equation}
where $d\in (0,1)$ and the term $o(1)$ converges to $0$ as $\varepsilon e^{-c}$ does so.
\end{lemma}
\begin{proof}
Set $f(z):= ze^z$ for $z\in\mr$. Changing the variable $z=-(y -b/a)$ transforms \eqref{eq:nd1} into an equivalent form
\begin{equation}\label{eq:tk1}
f(z)= -\varepsilon':= -(\varepsilon/a)e^{b/a} = -\varepsilon e^{-c},
\end{equation}
where $-\varepsilon' \in (-1/e,0)$ by assumption. According to Section 4 in \cite{Lambert}), equation \eqref{eq:tk1} has two solutions $z_1\in (-1,0)$ and
$z_2 =\log (-\varepsilon') - \log(- \log (-\varepsilon')) + o(1)$. Equivalently, equation \eqref{eq:nd1} has two solutions given in \eqref{eq:nd2}.
\end{proof}

\begin{proof}[Proof of Lemma \ref{lem:d1}]
According to \eqref{limit1} there exist constants $\delta, M>0$ such that
\begin{equation}\label{eq:ny2}
\delta   z(-\log z)\le 1- \phi(z) \le M z(-\log z), \quad z\in(0,1/2).
\end{equation}
Without loss of generality we assume that the first generation individuals are ordered $X_i\le X_{i+1}$ for all $i\in \N$, so that $X_1 = X_{\min}$.
Recall from Section \ref{LT} that $G$ is a nonnegative function given by $G(y) = e^y \E H(y)$ for $y\in\mr$, where
\begin{equation}\label{eq:H}
H(y) = \prod_{j=1}^N \phi(e^{-y-X_j}) - 1 +\sum_{j=1}^N(1-\phi(e^{-y-X_j}))\geq 0,\quad y\in\mr.
\end{equation}
Note that, for each $y\in\mr$, $H(y)\1_{\{N\leq 1\}}=0$ a.s. In view of this, in what follows we work on the event $\{N\geq 2\}$ without further notice.

For $y\geq 0$ and $\varepsilon\in (0, \min(\delta/e, 1/2, 1-\phi(1/e)))$ with the same $\delta$ as in \eqref{eq:ny2}, set
$A_y(\varepsilon):=\big\{\sum_{j=1}^N (1- \phi(e^{-y-X_j}))< \varepsilon\big\}$ and write
$$\int_0^\infty yG(y){\rm d}y =\E\int_0^\infty y e^y H(y)\1_{A_y(\varepsilon)}{\rm d}y
+\E \int_0^\infty y e^y H(y)\1_{(A_y(\varepsilon))^c}{\rm d}y =: I_1+I_2.$$

\noindent {\sc Proof of $I_1<\infty$.} We shall show that $I_1$ is finite under Condition $\mathcal{S}$. Fix any $y\geq 0$. Note that  $1-\phi(e^{-y-X_1})>\varepsilon$ provided that $y+X_1<1$, whence $A_y(\varepsilon)\subseteq A_y:=\{ y+X_i\ge 1~\text{for}~1\leq i\leq N\}$. We have, a.s.\ on  $A_y$,
\begin{equation}\label{eq:ny3}
\sum_{j=1}^N (1-\phi(e^{-y-X_j})) \le M \sum_{j=1}^N  e^{-y-X_j}(y+X_j) \le M e^{-y} (y (W_1^+ + W_1^-) + \widetilde W_1 + \widetilde W^-_1)
\end{equation}
and similarly
\begin{equation}\label{eq:ny4}
\sum_{j=1}^N (1-\phi(e^{-y-X_j}))\ge \delta  e^{-y} (y (W_1^+ + W_1^-) + \widetilde W_1 + \widetilde W^-_1),
\end{equation}
where $\widetilde W^-_1:=-\sum_{j=1}^N e^{-X_j}(X_j)_-$. Recall that the random variables $W_1^+$ and $W_1^-$ were defined in \eqref{eq:sob4}. Observe that $yW_1^- + \widetilde W_1^-\geq 0$ a.s.\ on $A_y$, although $\widetilde W^-_1\leq 0$ a.s.

On the event $A_y(\varepsilon)$, we have $\sum_{j=1}^N (1-\phi(e^{-y-X_j}))<\varepsilon<1/2$ a.s.\ which entails
$1-\phi(e^{-y-X_j})< 1/2$ for $j=1,2,\ldots, N$ a.s. In particular, using the inequality $-\log (1-z)\leq 2z$ for $z\in
[0,1/2]$ we obtain, a.s.\ on $A_y(\varepsilon)$,
\begin{equation}\label{aux2}
-\log \phi(e^{-y-X_j})\leq 2(1-\phi(e^{-y-X_j})),\quad j=1,2,\ldots, N
\end{equation}
and thereupon
$$\sum_{j=1}^N (-\log \phi(e^{-y-X_j})) \leq 2\sum_{j=1}^N(1-\phi(e^{-x-X_j}))<1.$$ An appeal to the inequality $e^{-z}\leq 1-z+z^2$ for $z\in
[0,1]$ enables us to conclude that, a.s.\ on $A_y(\varepsilon)$,
\begin{align*}
\prod_{j=1}^N \phi(e^{-y-X_i})&= \exp\bigg(\sum_{j=1}^N \log \phi(e^{-y-X_i}) \bigg)\\
&\leq 1+ \sum_{j=1}^N \log \phi(e^{-y-X_i})+ \Big(\sum_{j=1}^N \log  \phi(e^{-y-X_i})\Big)^2\\
&\leq 1-\sum_{j=1}^N(1-\phi(e^{-y-X_i}))+ \bigg(\sum_{j=1}^N(1-\phi(e^{-y-X_i}))\bigg)^2.
\end{align*}
Combining \eqref{eq:H} and \eqref{eq:ny3} yields
\begin{equation}\label{eq:nd3}
H(y)\1_{A_y(\varepsilon)} \le M^2 e^{-2y} \big(\widetilde W_1 + (yW_1^- + \widetilde W_1^-) + yW_1^+\big)^2\quad \text{a.s.}
\end{equation}
Further, we decompose $I_1$ into three parts depending on which of the terms $\widetilde W_1$, $yW_1^- + \widetilde W_1^-$
or $yW^+_1$ dominates. In view of \eqref{eq:ny4}, $A_{\varepsilon}(y)$ is a subset of each of the following sets $\{\widetilde W_1 <  \varepsilon_1 e^y\}$,
$\{yW_1^- + \widetilde W_1^- <  \varepsilon_1 e^y\}$ and $\{yW^+_1 <  \varepsilon_1 e^y\}$ for $\varepsilon_1:=\varepsilon/\delta$. Note that $\varepsilon_1 < 1/e$ by our choice of $\varepsilon$. The inequalities $y>-X_1$ a.s.\ on $A_y$ and \eqref{eq:nd3} together entail  $$I_1 \le 9  M^2 (I_{1,1} + I_{1,2} + I_{1,3}),$$ where
\begin{align*}
I_{1,1}: &= \E \widetilde W_1^2 \int_0^{\infty} y e^{-y}\1_{\{  \widetilde W_1 < \varepsilon_1 e^y   \}}{\rm d}y,\\
I_{1,2}: &= \E  \int_{-X_1}^{\infty} y e^{-y} (yW_1^- + \widetilde W_1^-)^2 \1_{A_y} \1_{\{ yW_1^- + \widetilde W_1^- < \varepsilon_1 e^y   \}}{\rm d}y,\\
I_{1,3}: &= \E ( W_1^+)^2 \int_0^{\infty} y^3 e^{-y} \1_{\{  y W^+_1 < \varepsilon_1 e^y   \}}{\rm d}y.
\end{align*}

As for $I_{1,1}$, write
\begin{multline*}
I_{1,1} \le \E \widetilde W_1^2 \int_{\log_+(\widetilde W_1/\varepsilon_1)}^\infty ye^{-y}{\rm d}y=\me \widetilde W_1^2(\log_+(\widetilde W_1/\varepsilon_1) + 1)e^{-\log_+(\widetilde W_1/\varepsilon_1)}\\ \leq \varepsilon_1 \me \widetilde W_1(\log(\widetilde W_1/\varepsilon_1) + 1)\1_{\{\widetilde W_1\geq \varepsilon_1\}}+\varepsilon_1^2<\infty.
\end{multline*}
Here, the finiteness is secured by \eqref{eqn:integrability} which is a part of Condition ${\mathcal S}$.

To deal with $I_{1,2}$, we intend to use Lemma \ref{lem:lambert} with $a=W_1^-$, $b=-\widetilde W_1^-$ and $\varepsilon=\varepsilon_1$. Since
\begin{equation}\label{eq:tk2}
- \frac{\widetilde W_1^-}{W_{1}^-} \le -X_1 \le \log W_{1}^-
\end{equation}
and $\varepsilon_1<1/e$, the lemma applies and justifies the inclusion
$\{y>0:\; \widetilde W_{1}^- + y W_{1}^- < \varepsilon_1 e^y\}\subseteq (0,Y_1)\cup(Y_2,\infty)$. Here, (random variables)
$Y_1$ and $Y_2$ are solutions to the equation
\begin{equation}\label{eq:sb1}
\widetilde W_{1}^- + y W_{1}^- = \varepsilon_1 e^y
\end{equation}
given by
\begin{equation}\label{eq:nd5}
  Y_1  = V - \widetilde W_1^-/ W_1^-,\qquad
  Y_2 = -\log\varepsilon_1 + \log W_1^- + \log \big( -\log \varepsilon_1 + \log W_1^- + \widetilde W_1^-/ W_1^- \big) + o(1),
\end{equation}
where $V$ is a nonnegative random variable bounded by $1$ a.s. In view of these observations we are going to consider the two integrals $I_{1,2}'$ and $I_{1,2}''$ with the integration sets being  $(0,Y_1)\cap(-X_1,\infty)$ and $(Y_2,\infty)$, respectively. Inequality \eqref{eq:tk2} tells us that $Y_1 \le -X_1 + V$ a.s., so that $(0,Y_1)\cap (-X_1,\infty) \subseteq (-X_1, -X_1 + V)$ and thereupon
\begin{multline*}
I'_{1,2} \le \E \int_{-X_1}^{-X_1+V} y e^{-y} (y W_1^- + \widetilde W_1^-)^2\1_{A_y}
\1_{\{ y W_1^- + \widetilde W_1^- <\varepsilon_1 e^y \}}{\rm d}y\\
\le \varepsilon_1^2 \E \int_{-X_1}^{-X_1+V} y e^y{\rm d}y \le \varepsilon_1^2 \E (-X_1+1) e^{-X_1+1} \le \varepsilon_1^2 e(\E W_1^- \log_+ W_1^-+\me W_1^-)<\infty.
\end{multline*}
Here, the finiteness is guaranteed by \eqref{eqn:integrability}. Further, recalling that $Y_2$ solves equation \eqref{eq:sb1} and changing the variable we obtain
\begin{align*}
I''_{1,2} &\le \E\int_{Y_2}^\infty y e^{-y} (yW_1^- + \widetilde W_1^-)^2{\rm d}y=
\E\int_0^\infty (y + Y_2) e^{-y}e^{-Y_2} (yW_1^- + Y_2 W_1^- + \widetilde W_1^-)^2 {\rm d}y\\&=  \E\int_0^\infty(y + Y_2) e^{-y}e^{-Y_2}
(yW_1^- + \varepsilon_1 e^{Y_2})^2{\rm d}y\le C \E (1+Y_2) \big( (W_1^-)^2 e^{-Y_2}+  W_1^- + e^{Y_2}\big).
\end{align*}
Here and hereafter, $C$ denotes a constant whose value is of no importance and may change from line to line. Using the inequalities
$$1+Y_2 \le C(1+\log W_1^-),\quad e^{Y_2} \le C W_1^- (C+ \log W_1^-),\quad
e^{-Y_2} \le C/W_1^-$$ which follow from \eqref{eq:nd5} we infer $I_{1,2}''\leq C \E (1+ \log W_1^- )^2 W_1^-$.
Condition ${\mathcal S}$ ensures that the right-hand side is finite.

Finally, to check that $I_{1,3}<\infty$ we proceed in the same way as above. One needs to determine precisely the integration domain, that is, to solve the equation $y W_{1}^+ = \varepsilon_1 e^y$. Lemma \ref{lem:lambert} (with $a=W_1^+$ and $b=0$) ensures the existence of two solutions to this equation: an a.s.\ bounded nonnegative random variable $Y_1$ and \begin{equation}\label{eq:nd4}
Y_2  = -\log\varepsilon_1 + \log W_1^+ + \log \big( -\log \varepsilon_1 + \log W_1^+ \big) + o(1).
\end{equation}
We skip further details. The proof of $I_1<\infty$ is complete.

\noindent {\sc Proof of $I_2<\infty $}. The function $G$ is bounded on $[0,1]$. Therefore, it is sufficient to consider the integral $I_2$ over the set $[1,\infty)$.
For $y\ge 1$, put $N^-(y):=\max\{j\leq N: y+X_j < 1\}$ with the standard convention that $N^-(y):=0$ if $y+ X_1 \ge 1$. Plainly, $N^-:=N^-(1)$ denotes the number of the first generation individuals located on the negative halfline, and, provided that $N^-\geq 1$, $X_{N^-}$ denotes the position of the rightmost first generation individual located on the negative halfline.

We shall use the inequality which follows directly from \eqref{eq:ny2} (compare \eqref{eq:ny3}):
\begin{equation}\label{xyz}
H(y) \le \sum_{j=1}^N (1-\phi(e^{-y-X_j}))\le Me^{-y} \big(\widetilde W_1 + y W_1^+ + F(y)\big)+ N^-(y),
\end{equation}
where $F(y):=\widetilde W_1(y) + y W_1(y)$,
\begin{equation}\label{eq:sob6}
\widetilde W_1(y):=\sum_{j=N^-(y)+1}^{N^-} e^{-X_j}X_j\quad\text{and}\quad
W_1(y):=\sum_{j=N^-(y)+1}^{N^-} e^{-X_j}.
\end{equation}
We note in passing that $F(y)=N^-(y)=0$ a.s.\ on $\{N^-=0\}$ and that, in general, $F(y)\geq 0$ a.s., but $\widetilde W_1(y) \le 0$ a.s. As we did before for $I_1$, we shall investigate the contribution of each term on the right-hand side of \eqref{xyz}
separately. To this end, we use the easily checked inequality $$(a_1+\ldots+a_m)\1_{\{a_1+\ldots+a_m>\rho\}}\leq m (a_1\1_{\{a_1>\rho/m\}}+\ldots+a_m\1_{\{a_m>\rho/m\}})$$ for $m\in\mn$, nonnegative $a_1,\ldots, a_m$ and $\rho>0$, to obtain $$I_2\leq 4 M (I_{2,1} + I_{2,2} + I_{2,3} + I_{2,4}).$$ Here, with $\varepsilon':=\varepsilon/(4M)$,
\begin{align*}
I_{2,1} & := \E \widetilde W_1 \int_1^\infty y\1_{\{\widetilde W_1 > \varepsilon' e^y \}}{\rm d}y,
\quad I_{2,2}:=\E W_1^+ \int_1^\infty y^2 \1_{\{ y W_1^+ > \varepsilon' e^y \}}{\rm d}y,\\
I_{2,3} &:= \E \int_1^\infty ye^y N^-(y){\rm d}y, \quad I_{2,4}:= \E \int_1^\infty yF(y)\1_{\{ F(y) > \varepsilon' e^y \}}{\rm d}y.
\end{align*}

The analysis of $I_{2,1}$ is simple: $$I_{2,1}\le \me\widetilde W_1 \int_0^{\log_+(\widetilde W_1/\varepsilon')} y {\rm d}y=(1/2)\E \widetilde W_1 (\log_+ (\widetilde W_1/\varepsilon'))^2<\infty,$$ where the finiteness is a consequence of the second part of \eqref{eq:2.4}.

To treat $I_{2,2}$ we use the same $Y_2$ as in \eqref{eq:nd4} which gives $$I_{2,2}\le C+ C \E \1_{\{W^+_1 > \varepsilon'\}} W^+_1 \int_0^{Y_2} y^2{\rm d}y\le C(1+\E W^+_1 (\log_+ W^+_1)^3)<\infty.$$ Here, the finiteness is justified by the first part of \eqref{eq:2.4}.

Next, we work with $I_{2,3}$. For notational simplicity, let $X_0: = -\infty$ and $X_{N^-+1}:=0$. Put $g(y)=(y-1)e^y$ for $y\in\mr$ and note that $g'(y)=ye^y$.
Since $N^-(y)=j$ for $y\in (-X_{j+1},-X_j)$, we have
\begin{align*}
I_{2,3}&= \E \sum_{j=0}^{N^-} j \int_{-X_{j+1}+1}^{-X_j+1}g^\prime(y){\rm d}y=\E \sum_{j=1}^{N^-} j (g(-X_j+1) - g(-X_{j+1}+1))\\
& = \E \sum_{j=1}^{N^-}  g(-X_{j}+1)
\le e \E W_1^- \log_+ W_1^-<\infty.
\end{align*}
The finiteness follows from \eqref{eqn:integrability}.

It remains to prove that $I_{2,4}<\infty$. Put
$B:=\big\{\sum_{j=1}^{N^-} e^{-\Delta_j}(1+\Delta_j) \le C_0  \big\}$,  where $\Delta_j:= X_j - X_1$ and $C_0$ is the same as in \eqref{eq:2.5}. Write
\begin{multline*}
I_{2,4} \le \E \1_B \int_0^{\max(-X_1,1)} y F(y){\rm d}y+\E \1_{B^c} \int_0^{\max(-X_1,1)} y F(y){\rm d}y\\
+\E \1_B \int_{\max(-X_1,1)}^\infty y F(y)\1_{\{F(y)>\varepsilon' e^y\}}{\rm d}y+ \E \1_{B^c} \int_{\max(-X_1,1)}^\infty F(y)\1_{\{F(y)>\varepsilon'e^y\}}{\rm d}y =: J_1+J_2+J_3+J_4.
\end{multline*}
For $y\ge 1$, we have, a.s.\ on $B\cap \{y <-X_1\}$, $$F(y)=\sum_{j=1}^{N^-}e^{-X_j}(y+X_j)_+ \le e^{-X_1} \sum_{j=1}^{N_-}e^{-\Delta_j}\Delta_j\le C_0 e^{-X_1}$$ and thereupon
$$J_1 \le (C_0/2) \E (-X_1)^2 e^{-X_1} \le (C_0/2)\E W_1^- (\log W_1^-)^2 <\infty.$$ Here, the finiteness is ensured by \eqref{eqn:integrability}. Next, using $(-X_1)\le W_1^-$ a.s.\ we infer
\begin{equation*}
J_2 \le \E W_1^- \1_{B^c}\int_0^{-X_1} y^2{\rm d}y \le (1/3) \E W_1^-\1_{B^c}(-X_1)^3 \le (1/3)\E W_1^- (\log W_1^-)^3\1_{B^c}<\infty,
\end{equation*}
where the finiteness is a consequence of \eqref{eq:2.4}.

It holds, a.s.\ on $B$, that
$$ e^{-y+X_1}F(y-X_1) =  \sum_{j=1}^{N^-}e^{-y-\Delta_j}(y+\Delta_j)=e^{-y}y\sum_{j=1}^{N^-}e^{-\Delta_j}+e^{-y}\sum_{j=1}^{N^-}e^{-\Delta_j}\Delta_j\le C_0 (y+1) e^{-y}.$$
With this at hand, we obtain
\begin{multline*}
J_3 \le  (1/\varepsilon') \E \1_B \int_{\max(-X_1,1)}^\infty y e^{-y} (F(y))^2{\rm d}y=(1/\varepsilon')\E \1_B \int_{\max(0, 1+X_1)}^\infty (y-X_1) e^{-y+X_1}(F(y-X_1))^2{\rm d}y\\ \leq C_0^2\E e^{-X_1}\int_0^\infty (y-X_1)( y+1)^2 e^{-y}{\rm d}y \le C\E(1+ (-X_1)e^{-X_1})\le C\E ( 1+ W_1^- \log_+W_1^-)<\infty,
\end{multline*}
where the finiteness is secured by \eqref{eqn:integrability}.

Finally, to deal with $J_4$ we denote by $Y_2$ the larger solution to the equation $yW_1^- = \varepsilon' e^y$. According to Lemma \ref{lem:lambert},
$Y_2 =-\log \varepsilon' + \log W_1^- + \log(-\log \varepsilon' + \log W_1^-) + o(1)$ which entails
\begin{multline*}
J_4 \le \E\1_{B^c}\int_{-X_1}^\infty y^2 W_1^- \1_{\{yW_1^-  > \varepsilon' e^y\}}{\rm d}y\\
 \le C + C \E  \1_{B^c}  \1_{\{ W^-_1 > \varepsilon'\}} W^-_1 \int_{-X_1}^{Y_2} y^2{\rm d}y
\le C \E \1_{B^c} W_1^- (\log W_1^-)^3 <\infty.
\end{multline*}
The finiteness is ensured by \eqref{eq:2.5}. The proof of $I_2<\infty$ is complete.
\end{proof}

\begin{lemma}\label{suff1}
Assume that Condition $\mathcal{S}$ holds. Then \eqref{eq:2.4} is necessary for $\int_0^{\infty}y G(y){\rm d}y<\infty$.
\end{lemma}
\begin{proof}
In view of \eqref{limit1} there exists $\delta_1>0$ such that $$1-\phi(z) \ge \delta_1 z (-\log z)_+, \quad z\ge 0.$$ Hence, for $y\geq 0$,
\begin{equation}\label{123}
\sum_{j=1}^N(1-\phi(e^{-y-X_j}))\ge \delta_1 \sum_{j=1}^N e^{-y-X_j}(y+X_j)_+ \ge \delta_1 e^{-y}\big( \widetilde W_1 + yW_1^+ \big).
\end{equation}
For each $y>0$, define the event $D_y:=\{\delta_1 e^{-y}\big( \widetilde W_1 + yW_1^+ \big) > 2 \}$. If $D_y=\oslash$ for all $y>0$, then both $\widetilde W_1$ and $W_1^+$ are a.s.\ bounded random variables which entails that \eqref{eq:2.4} holds. Thus, in what follows we assume that $D_y\neq \oslash$ for some $y>0$. For such $y$, we conclude with the help of \eqref{123} that, a.s.\ on $D_y$,
\begin{equation}\label{1234}
\prod_{j=1}^N \phi(e^{-y-X_j})-1+\sum_{j=1}^N (1-\phi(e^{-y-X_j})) \ge -1+\sum_{j=1}^N (1-\phi(e^{-y-X_j}))\geq (\delta_1/2)e^{-y}\big( \widetilde W_1 + yW_1^+ \big).
\end{equation}
This in combination with the inclusions $\{ \widetilde W_1 > 2e^y/\delta_1 \}\subseteq D_y$ and $\{ y W^+_1 > 2e^y/\delta_1\}\subseteq D_y$  yields
\begin{multline*}
\infty>\int_0^\infty yG(y){\rm d}y \ge (\delta_1/2) \E \int_0^\infty y\big(\widetilde W_1 + yW_1^+ \big)\1_{D_y}{\rm d}y\\
\geq (\delta_1/2)\Big(\E \widetilde W_1 \int_0^\infty y\1_{\{\widetilde W_1 > 2e^y/\delta_1\}}{\rm d}y
+\E  W^+_1 \int_1^\infty y^2\1_{\{W^+_1>2e^y/\delta_1\}}{\rm d}y\Big)\\ = (\delta_1/4)\E \widetilde W_1 \big(\log_+(\delta_1 \widetilde W_1/2)\big)^2+
(\delta_1/6)\big(\E W^+_1 \big(\log((\delta_1 W^+_1/2)\vee e)\big)^3-\me W^+_1\big).
\end{multline*}
This proves the necessity of \eqref{eq:2.4}.
\end{proof}

\begin{lemma}\label{lem:6.13}
Assume that Condition $\mathcal{S}$ and \eqref{mom3} hold. Then \eqref{eq:2.5} is necessary for $\int_0^\infty yG(y)dy <\infty$. 
\end{lemma}
\begin{proof}
We retain, for the most part, the notation from the proof of Lemma \ref{lem:d1}.  Additionally, we put $B:=\big\{\sum_{j=1}^{N^-}e^{-\Delta_j}(1+\Delta_j) \le 2e/\delta_1 \big\}$ (with $\delta_1$ as in \eqref{123}) and, for $y>0$,
$D_y: = \{\delta_1 e^{-y} F(y) >2 \}$. Assume that $D_y=\oslash$ for all $y>0$. Then taking $y=(-X_1)+1$ we conclude that $\P(B)=1$ which implies that \eqref{eq:2.5} holds with any $C_0>2e/\delta_1$. Therefore, from now on we assume that $D_y\neq \oslash $ for some $y>0$. By the argument leading to \eqref{1234}, we have, a.s.\ on $D_y$,
$ H(y) \ge (\delta_1/2) e^{-y} F(y)$, whence $$\infty > \int_0^\infty yG(y){\rm d}y \ge (\delta_1/2) \E \int_0^\infty yF(y)\1_{D_y}{\rm d}y.$$ This particularly yields
\begin{equation}\label{eq:el2}
I_1:=\E \int_0^{-X_1} y F(y)\1_{D_y}{\rm d}y <\infty \quad \mbox{and} \quad
I_2: = \E \int_{-X_1+1}^\infty y F(y) \1_{D_y}\1_{B^c}{\rm d}y <\infty.
\end{equation}

We first prove that $I_1<\infty$ entails
\begin{equation}\label{eq:el3}
\E (-X_1)^3 W_1^- < \infty.
\end{equation}
Indeed, observe that
\begin{align*}
I_1^\ast:= \E \int_0^{-X_1} y F(y)\1_{D_y^c}{\rm d}y &=
\E \int_0^{-X_1} y F(y)\1_{\{F(y)\leq (2/\delta_1) e^y\}}{\rm d}y\\
&\le (2/\delta_1) \E \int_0^{-X_1} y e^y {\rm d}y=(2/\delta_1)(\E (-X_1) e^{-X_1}-\me e^{-X_1}+1)<\infty
\end{align*}
as a consequence of \eqref{eqn:integrability}. Summing up $I_1$ and $I_1^\ast$ we obtain
\begin{align*}
\infty & > \E \int_0^{-X_1} y F(y){\rm d}y=\E \int_0^{-X_1} y \sum_{j=1}^{N^-}e^{-X_j}(y+X_j)_+ {\rm d}y\\
& = \E \sum_{j=1}^{N^-} e^{-X_j}\int_{-X_j}^{-X_1} y (y+X_j){\rm d}y\\
& = \E \sum_{j=1}^{N^-}e^{-X_j}\Big[(1/3)((-X_1)^3 -(-X_j)^3) + (1/2) ( (-X_1)^2 - (-X_j)^2)X_j\Big]\\
  & = (1/6)\E \sum_{j=1}^{N^-}e^{-X_j}\Big[ 2(-X_1)^3 + (-X_j)^3 - 3 (-X_1)^2 (-X_j)\Big]\\
  & = (1/6) \E \sum_{j=1}^{N^-}e^{-X_j} \Delta_j^2  (2(-X_1) + (-X_j)),
\end{align*}
where we have used the identity $2a^3+b^3 - 3a^2b = (a-b)^2(2a+b)$, $a,b\in\mr$ for the last equality. Hence,
\begin{equation}\label{eq:sob1}
\E \sum_{j=1}^{N^-}e^{-X_j}(-X_j)\Delta_j^2 <\infty.
\end{equation}
The inequality $a^3 \le 8b^3 + 4a(a-b)^2$ holds for any $a>b>0$. Using it with $a=-X_1$ and $b=-X_i$ we infer
$$ \E (-X_1)^3 \sum_{j=1}^{N^-} e^{-X_j} \le 8 \sum_{j=1}^{N^-}e^{-X_j}(-X_j)^3 + 4 \sum_{j=1}^{N^-}e^{-X_j}(-X_j)\Delta_j^2.$$
This reveals that \eqref{eq:el3} is a consequence of \eqref{mom3} and \eqref{eq:sob1}.

After these preparations we are ready to show the necessity of condition \eqref{eq:2.5}. To this end, we first observe that, a.s.\ on $B^c$, $y_0W_1^- + \widetilde W_1^- > (2/\delta_1) e^{y_0}$, where $y_0:=-X_1+1$. This implies that $D_y \cap (-X_1+1 ,\infty)=(-X_1+1,Y_2)$ with $Y_2$ being the larger solution to the equation $yW_1^- + \widetilde W_1^-= (2/\delta_1) e^y$. Recall that  the $Y_2$ is given by \eqref{eq:nd5} with $2/\delta_1$ replacing $\varepsilon_1$. As a consequence, we obtain
\begin{align*}
\infty & > I_2=\E \1_{B^c}\int_{-X_1+1}^{Y_2} y(yW^-_1 + \widetilde W^-_1){\rm d}y\\
&= (1/6) \E\1_{B^c} \Big[2W_1^-(Y_2^3 - (-X_1+1)^3) + 3 \widetilde W^-_1(Y_2^2 - (-X_1+1)^2)\Big],
 \end{align*}
and inequality \eqref{eq:el3} ensures that
$$\infty > J:= \E\1_{B^c} \Big[ 2W^-_1 Y_2^3+ 3 \widetilde W_1^- Y_2^2 \Big].$$ Put $A:=\{\log W_1^- < 2(-X_1)\}$. We have, a.s.\ on $B^c\cap A$,
\begin{equation}\label{eq:sob2}
  \big| 2W^-_1 Y_2^3+ 3 \widetilde W_1^- Y_2^2 \big| \le C
\big(W^-_1 (\log W_1^-)^3+ (-X_1) W_1^- (\log W_1^-)^2\big)\le C  W_1^- (-X_1)^3,
\end{equation}
whereas, a.s.\ on $B^c\cap A^c$,
\begin{equation}\label{eq:sob3}
\begin{split}
2W^-_1 Y_2^3+ 3 \widetilde W_1^- Y_2^2 &\ge
2W^-_1 Y_2^3+ 3 X_1 W_1^- Y_2^2 \ge
W_1^- Y_2^2 (2Y_2 + 3X_1)\\ &\ge
W_1^- Y_2^2 (2\log W_1^- + 3X_1) \ge
(1/2) W_1^- (\log W_1^-)^3.
\end{split}
\end{equation}
Combining \eqref{eq:sob2} and \eqref{eq:sob3} yields
\begin{align*}
\infty & > 2J \ge \E \1_{B^c}\1_{A^c} W_1^-(\log W_1^-)^3 - C \E W_1^- (-X_1)^3\\
& = \E \1_{B^c} W_1^-(\log W_1^-)^3  - \E \1_{B^c}\1_{A} W_1^-(\log W_1^-)^3-C \E W_1^- (-X_1)^3\\
& \ge \E \1_{B^c} W_1^-(\log W_1^-)^3-(C+8)\E W_1^- (-X_1)^3.
\end{align*}
Invoking \eqref{eq:el3} we conclude that condition \eqref{eq:2.5} holds.
\end{proof}

Now we are ready to prove Theorem \ref{expaLT}.
\begin{proof}[Proof of Theorem \ref{expaLT}]
We first note that, by Lemma \ref{don} (b), the distribution of $S_{\tau_1}$ is nonarithmetic.

In view of Theorem \ref{impo} and Lemma \ref{finite}, the first equality in \eqref{limit123} holds if, and only if, $\int_0^\infty yG(y){\rm d}y<\infty$. Thus, the second equality in \eqref{limit123} holds if, and only if, $\int_0^\infty yG(y){\rm d}y<\infty$ and
\begin{equation}\label{second}
\lim_{x\to\infty}(\mu U(x)-x)=c_3
\end{equation}
for some finite constant $c_3$. By Lemma \ref{don} (d), relation \eqref{second} holds if, and only if, $\me \xi_-^3<\infty$ (which is nothing else but \eqref{mom3}).

Now we conclude with the help of Lemmas \ref{suff1} and \ref{lem:6.13} that the second equality in \eqref{limit123} also entails \eqref{eq:2.4} and  \eqref{eq:2.5}, hence
Condition $\mathcal{S}^\ast$. Sufficiency of \eqref{eq:2.4} and  \eqref{eq:2.5} for the first equality in \eqref{limit123} is justified by Lemma \ref{lem:d1}.
\end{proof}

\section{Proofs related to the rate of convergence}
\subsection{Auxiliary results}

We start with a few auxiliary facts which can be lifted from the existing literature.
\begin{lemma}\label{aux_limit}
Assume that Condition $\mathcal{S}$ holds. Then,

\noindent (a) $n^{1/2}\sum_{|u|=n}e^{-S(u)}\overset{\mmp}{\to}
(2/(\pi \sigma^2))^{1/2}Z$ as $n\to\infty$;

\noindent (b) $M^\ast_n:=\inf_{|u|=n}S(u)-2^{-1}\log
n~\overset{\mmp}{\to}~ +\infty$ as $n\to\infty$;

\noindent (c) for $\beta>1$ and $m\in\mn$,
$$n^{\beta/2}\sum_{|u|=n}e^{-\beta S(u)}(S(u)-2^{-1}\log n)^m~\overset{\mmp}{\to}~0,\quad n\to\infty.$$
\end{lemma}
\begin{proof} (a) This is Theorem 1.1 in \cite{Aidekon+Shi:2014}.

\noindent (b) This follows from Theorem 1.1 in \cite{Mallein:2016} which states that the sequence of distributions of $(M_n^\ast-\log n)_{n\in\mn}$ is tight. Noting that
the cited result considers maxima rather than minima we refer to Lemma A.1 in \cite{Mallein:2018} for a proof of the fact that the assumptions imposed in \cite{Mallein:2016} are equivalent to Condition $\mathcal{S}$. 

\noindent (c) Let $\alpha>1$. The sequence of distributions of
$(\sum_{|u|=n}e^{-\alpha (S(u)-(3/2)\log n)})_{n\in\mn}$ is tight by Proposition~2.1 in
\cite{Madaule:2017}.
This implies that
\begin{equation}\label{x1}
\sum_{|u|=n}e^{-\alpha (S(u)-2^{-1}\log
n)}~\overset{\mmp}{\to}~0,\quad n\to\infty.
\end{equation}
Pick any $\varepsilon\in (0,\beta-1)$. Then, for all $x>0$,
$e^{-\beta x}x^m\leq \varepsilon^{-m}m!e^{-(\beta-\varepsilon)x}$.
Using this we obtain, for any $\delta>0$,
\begin{multline*}
\mmp\Big\{\Big|\sum_{|u|=n}e^{-\beta(S(u)-2^{-1}\log
n)}(S(u)-2^{-1}\log n)^m\Big|>\delta\Big\}\\\leq
\mmp\Big\{\sum_{|u|=n}e^{-\beta(S(u)-2^{-1}\log
n)}(S(u)-2^{-1}\log n)^m>\delta, M_n^\ast>0\Big\}+\mmp\{M_n^\ast\leq 0\}\\\leq
\mmp\Big\{\sum_{|u|=n}e^{-(\beta-\varepsilon)(S(u)-2^{-1}\log
n)}>\delta \varepsilon^m/m!\Big\}+\mmp\{M_n^\ast\leq 0\}.
\end{multline*}
Sending $n\to\infty$ we conclude that each summand on the
right-hand side converges to $0$ by \eqref{x1} and part (b) of the
lemma, respectively.
\end{proof}

\subsection{Proof of Theorem \ref{main2}}\label{proofs2}

Put $H(x):=\me Z\1_{\{Z\leq x\}}$ for $x\geq 0$ and let $L^\ast$ be a random
variable which is independent of $\F_\infty$ and has a $1$-stable distribution with the generating triple
$((1-\gamma)(2/(\pi\sigma^2))^{1/2}, (\pi/(2\sigma^2))^{1/2}, 1)$. Note that $L$ has the same distribution as $L^\ast+(2/(\pi\sigma^2))^{1/2}c$.

Only assuming that Condition $\mathcal{S}_{{\rm na}}$ and \eqref{Z} hold we shall prove more general results
\begin{equation}\label{limit_main000}
\me\Big(f\Big(n^{1/2}\Big(Z-\sum_{|u|=n}e^{-S(u)}H(e^{S(u)}n^{-1/2})\Big)\Big)\Big|\F_n\Big)~\overset{\mmp}{\to}~ \E(f(ZL^\ast)|\F_\infty),\quad n\to\infty
\end{equation}
and
\begin{equation}\label{limit_main}
n^{1/2}\Big(Z-\sum_{|u|=n}e^{-S(u)}H(e^{S(u)}n^{-1/2})\Big)~\dod~
ZL^\ast,\quad n\to\infty,
\end{equation}
and then obtain \eqref{limit_main01} and \eqref{limit_inter2} as corollaries. Our argument is based on the following representation
\begin{align*}
  \Theta_n&:=n^{1/2}\Big(Z-\sum_{|u|=n}e^{-S(u)}H(e^{S(u)}n^{-1/2})\Big)\\
  &=n^{1/2}\sum_{|u|=n}e^{-S(u)}\big(Z(u)-H(e^{S(u)}n^{-1/2})\big),\quad
n\in\mn\quad\text{a.s.}
\end{align*}
which follows from \eqref{fixpoint}.

In view of Lemma \ref{aux_limit}, from any deterministic increasing sequence which diverges to $\infty$ we can extract
a subsequence $(n_k)_{k\in\mn}$ such that
\begin{equation}\label{x2}
\lim_{k\to\infty}(\inf_{|u|=n_k}S(u)-2^{-1}\log
n_k)=+\infty\quad\text{a.s.};
\end{equation}
\begin{equation}\label{x3}
\lim_{k\to\infty}n_k^{1/2}\sum_{|u|=n_k}e^{-S(u)}=(2/(\pi\sigma^2))^{1/2}
Z\quad \text{a.s.};
\end{equation}
for $m=1,2$,
\begin{equation}\label{x4}
\lim_{k\to\infty}n_k\sum_{|u|=n_k}e^{-2S(u)}(S(u)-2^{-1}\log
n_k)^m=0\quad\text{a.s.}
\end{equation}

For $n\in\mn_0$ and the $\sigma$-algebra $\F_n$ defined in Section \ref{intro}, we shall use the following notation $\mmp_n\{\cdot\}:=\mmp\{\cdot|\F_n\}$ and, for a random variable $\theta$, $\me_n\theta:=\me(\theta|\F_n)$ and ${\rm Var}_n\theta:={\rm Var}\,(\theta|\F_n)=\me (\theta^2|\F_n)-(\me (\theta|\F_n))^2$.

Suppose we can check that the triangular array
$$(T_{u,k})_{|u|=n_k,\,k\in\mn}:=\big(n_k^{1/2}e^{-S(u)}\big(Z(u)-H(e^{S(u)}n_k^{-1/2})\big)\big)_{|u|=n_k,\,k\in\mn}$$ is a null array, that is, for every $\delta>0$,
\begin{equation}\label{null}
\lim_{k\to\infty}\sup_{|u|=n_k}\mmp_{n_k}\big\{|T_{u,k}|>\delta\big\}=0\quad \text{a.s.},
\end{equation}
that, for every $x>0$,
\begin{align}
M(x) &:= -\lim_{k\to\infty}\,\sum_{|u|=n_k}
\mmp_{n_k}\big\{T_{u,k} > x\big\}=-(2/(\pi\sigma^2))^{1/2}Zx^{-1} \quad \text{a.s.} \label{eq:L(x)>0}   \\
M(-x) &:= \lim_{k\to\infty}\,\sum_{|u|=n_k} \mmp_{n_k}\big\{
T_{u,k}\leq -x\big\}=0 \quad \text{a.s.}; \label{eq:L(x)<0}\\
\sigma^2&:=\lim_{\varepsilon \to 0+} \lim_{k\to\infty}\,
\sum_{|u|=n_k} {\rm Var}_{n_k}\big[T_{u,k}\1_{\{|T_{u,k}|\le \varepsilon\}}\big]=0
\quad   \text{a.s.} \label{eq:sigma}
\end{align}
and, for every $\tau>0$,
\begin{equation}   \label{eq:a}
a_0(\tau):= \lim_{k\to\infty}\,\sum_{|u|=n_k}
\me_{n_k}\big[T_{u,k}\1_{\{|T_{u,k}|\le \tau\}}\big] = (2/(\pi \sigma^2))^{1/2}Z \log \tau \quad \text{a.s.}
\end{equation}
Then, according to Theorem 1 on p.~116 in \cite{Gnedenko+Kolmogorov:1968},
\begin{multline}
\lim_{k\to\infty} \me_{n_k}\big[{\rm i} t \Theta_{n_k}\big]=
\exp\bigg({\rm i} at-\frac{\sigma^2t^2}{2}+\int_{\mr \setminus \{0\}}\bigg(e^{{\rm i}tx}-1-\frac{{\rm i}tx}{1+x^2}\bigg) \, {\rm d}M(x)\bigg)\\
=\exp\bigg((2/(\pi\sigma^2))^{1/2} Z \int_0^\infty \bigg(e^{{\rm
i}tx}-1-\frac{{\rm i}tx}{1+x^2}\bigg)x^{-2}\, {\rm d}x
\bigg)\\=\exp\big(Z({\rm
i}(1-\gamma)(2/(\pi\sigma^2))^{1/2}t-(\pi/(2\sigma^2))^{1/2}|t|(1+{\rm i}\,{\rm
sgn}\,(t)(2/\pi)\log |t|)\big)\quad \text{a.s.}
\label{eq:Kolmogorov's limit relation}
\end{multline}
for $t\in\mr$, where $\gamma$ is the Euler-Mascheroni constant.
Here,
\begin{multline*}
a:= a_0(\tau)-\int_{[-\tau,\,\tau]}\frac{x^3}{1+x^2} \,{\rm d}L(x)
+\int_{\mr \setminus [-\tau,\,\tau]}\frac{x}{1+x^2} \, {\rm
d}L(x)\\ =(2/(\pi\sigma^2))^{1/2} Z\Big(\log \tau-\int_0^\tau \frac{x}{1+x^2} \,{\rm
d}x +\int_\tau^\infty \frac{1}{x(1+x^2)} \, {\rm d}x\Big)=0,
\end{multline*}
and the last equality in \eqref{eq:Kolmogorov's limit relation} follows from calculations given on p.\;170
in \cite{Gnedenko+Kolmogorov:1968}. However, the constant
$1-\gamma$ is not given explicitly in
\cite{Gnedenko+Kolmogorov:1968} and rather represented as the
integral $$\Gamma:=\int_0^\infty\Big(\frac{\sin
x}{x^2}-\frac{1}{x(1+x^2)}\Big){\rm d}x.$$ To evaluate it, write
$$\Gamma=\int_0^\infty\Big(\frac{\sin
x}{x^2}-\frac{1}{x(1+x)}\Big){\rm
d}x+\int_0^\infty\Big(\frac{1}{x(1+x)}-\frac{1}{x(1+x^2)}\Big){\rm
d}x.$$ While the first integral is equal to $1-\gamma$ by formula
(3.781.1) in \cite{Gradshteyn+Ryzhik:2007}, the second is equal to
$0$ which can be seen by direct calculation. Equivalently, we have shown that, for every bounded continuous function $f:\mr \to \mr$,
\begin{equation*}
\lim_{k\to\infty} \me\big(f(n_k^{1/2}\Big(Z-\sum_{|u|=n_k}e^{-S(u)}H(e^{S(u)}n_k^{-1/2})\Big)\big |\F_{n_k}\big)=\me \big(f(ZL^\ast)|\F_\infty\big)\quad  \text{ a.s.}
\end{equation*}
which, by a standard argument, entails \eqref{limit_main000}.

To obtain distributional convergence \eqref{limit_main} just observe that \eqref{limit_main000} and the Lebesgue dominated convergence theorem guarantee $$\lim_{n\to\infty} \me f\Big(n^{1/2}\Big(Z-\sum_{|u|=n}e^{-S(u)}H(e^{S(u)}n^{-1/2})\Big)\Big |\F_n\Big)=\me f(ZL^\ast)$$ which is equivalent to \eqref{limit_main}.

Thus, we are left with proving \eqref{null} through \eqref{eq:a}. As a preparation, denote by $F(x):=\mmp\{Z\leq x\}$ for $x\in \R$, the distribution
function of $Z$, and recall that $Z$ is nonnegative random variable, whence $F(x)=0$ for $x<0$. Condition \eqref{Z} reads
\begin{equation}\label{x5}
\lim_{t\to\infty} t(1-F(t))=1.
\end{equation}
Further, by Lemma \ref{link1}, relation \eqref{x5} is equivalent to
the following: for each $\lambda>0$,
\begin{equation}\label{x6}
\lim_{t\to\infty} (H(\lambda t)-H(t))=\log\lambda
\end{equation}
and implies that
\begin{equation}\label{Hlog}
H(t)\sim \log t,\quad t\to\infty
\end{equation}
(alternatively, \eqref{Hlog} also holds by Theorem \ref{expa0}).

For any $z\in\R$ and $u$ with $|u|=n_k$, put
$$a(z,u,k):=ze^{S(u)} n_k^{-1/2} + H(e^{S(u)}n_k^{-1/2}),$$ so that $\{T_{u,k}>z\} = \{Z(u)>a(z,u, k)\}$. As a consequence of \eqref{x2} and $\lim_{x\to\infty}x^{-1}H(x)=0$ (use \eqref{Hlog} for the latter), the first term of $a(z,u, k)$ dominates which entails $\lim_{k\to\infty}{\rm sgn}(z)a(z,u,k)=+\infty$ a.s.\ for $z\neq 0$. Using \eqref{x5} and independence of $Z(u)$ for $u$ with $|u|=n_k$ and $\F_{n_k}$ we obtain, for $z>0$,
\begin{equation}\label{eq:ptu}
\P_{n_k}\big\{T_{u,k} >z\big\}=1- F(a(z,u,k))~ \sim~ a(z,u,k)^{-1}~ \sim~ z^{-1} n_k^{1/2} e^{-S(u)}\quad\text{a.s.}
\end{equation}
as $k\to \infty$.
By a similar reasoning, for $z>0$, $u$ with $|u|=n_k$ and large enough $k$,
$$F(a(-z,\varepsilon,k))=0\quad \text{a.s.} \quad \mbox{and}
\quad \1_{\{|T_{u,k}| \leq z\}}=\1_{\{T_{u,k} \leq z\}}\quad\text{a.s.}$$
We shall repeatedly use these observations, without further notice.

\noindent {\sc Proof of \eqref{null}}. In view of \eqref{eq:ptu}, for each $\delta>0$,
$$\sup_{|u|=n_k}\mmp_{n_k}\big\{|T_{u,k}|>\delta\big\} ~\sim~ \sup_{|u|=n_k}\delta^{-1} e^{-S(u)}n_k^{1/2}=\delta^{-1}\exp(-\inf_{|u|=n_k}(S(u)-2^{-1}\log n_k))\quad \text{a.s.}$$ as $k\to\infty$. By \eqref{x2}, the right-hand side converges to $0$ a.s.\ as $k\to\infty$ which proves \eqref{null}.

\noindent {\sc Proofs of \eqref{eq:L(x)>0} and \eqref{eq:L(x)<0}}.
By another appeal to \eqref{eq:ptu}, for any $x> 0$,
$$\sum_{|u|=n_k} \mmp_{n_k}\{T_{u,k}>x\}~ \sim~x^{-1} n_k^{1/2}\sum_{|u|=n_k}e^{-S(u)}~\to~(2/(\pi\sigma^2))^{1/2}Zx^{-1}\quad \text{a.s.}$$ as $k \to \infty$ which proves \eqref{eq:L(x)>0}. Here, the limit relation is a consequence of \eqref{x3}. The proof of
\eqref{eq:L(x)<0} is easy: for large enough $k$,
$$\sum_{|u|=n_k} \mmp_{n_k}\big\{T_{u,k} \le -x  \big\}= \sum_{|u|=n_k}F(a(-x, u, k))=0 \quad \text{a.s.}$$

\noindent {\sc Proof of \eqref{eq:sigma}}. First, note that
according to Theorem 1.6.4 in \cite{Bingham+Goldie+Teugels:1989},
relation \eqref{x5} entails
\begin{equation}\label{asy_H2}
H_2(t):=\me Z^2\1_{\{Z\leq t\}}~\sim~ t,\quad t\to\infty.
\end{equation}
For $\varepsilon > 0$ and large enough $k$,
\begin{multline*}
\sum_{|u|=n_k} {\rm Var}_{n_k}\! \big[T_{u,k}\1_{\{|T_{u,k}| \leq \varepsilon\}} \big]
\leq \sum_{|u|=n_k}\me_{n_k}\! \big[n_k
e^{-2S(u)}(Z(u)-H(e^{S(u)}n_k^{-1/2}))^2\1_{\{Z(u)\leq a(\varepsilon,u,k)\}} \big]\\
\leq~ n_k
\sum_{|u|=n}e^{-2S(u)}\Big(H_2\big(a(\varepsilon,u,k)\big)-2H\big(e^{S(u)}n_k^{-1/2}\big) H\big(a(\varepsilon,u,k)\big)+H^2(e^{S(u)}n_k^{-1/2})\Big)\\=:I_1(n_k)-2I_2(n_k)+I_3(n_k).
\end{multline*}
For the first inequality we have used the fact that the
conditional variance does not exceed the conditional second moment.
Further, we investigate each term $I_j(n_k)$, $j=1,2,3$
separately. By \eqref{asy_H2}, \eqref{x2} and \eqref{Hlog}, as $k\to\infty$,
$$I_1(n_k)~\sim~ n_k \sum_{|u|=n_k}e^{-2S(u)}a(\varepsilon,u,k)~\sim~ \varepsilon
n_k^{1/2}\sum_{|u|=n_k} e^{-S(u)}\quad\text{a.s.}$$ According to \eqref{x3}, the
last expression converges to $\varepsilon (2/(\pi\sigma^2))^{1/2}Z$ a.s.\ as
$k\to\infty$ which, in its turn, converges to $0$ a.s.\ as $\varepsilon\to 0+$.
By \eqref{x2} and \eqref{Hlog}, as $k\to\infty$, $$I_2(n_k)~\sim~n_k \sum_{|u|=n}e^{-2S(u)}(S(u)-2^{-1}\log
n_k)(\log \varepsilon+S(u)-2^{-1}\log n_k)\quad\text{a.s.}$$ In view of
\eqref{x4}, this and $I_3(n_k)$ converge to $0$ a.s.\ as
$k\to\infty$. The proof of \eqref{eq:sigma} is complete.

\noindent {\sc Proof of \eqref{eq:a}}. For each $\tau > 0$ and large $k$,
\begin{multline*}
\sum_{|u|=n_k} \me_{n_k}\big[T_{u,k}\1_{|T_{u,k}|\leq \tau\}} \big]=
n_k^{1/2} \sum_{|u|=n_k}e^{-S(u)}\big(H(a(\tau, u,k))-H(e^{S(u)}n_k^{-1/2})\big)\\+
n_k^{1/2} \sum_{|u|=n_k}e^{-S(u)}H(e^{S(u)}n_k^{-1/2})\big(1-F(a(\tau,u,k))\big):=J_1(n_k)+J_2(n_k).
\end{multline*}
Arguing as in the proof of \eqref{eq:L(x)>0} we conclude that, as
$k\to\infty$,
$$J_2(n_k)~\sim~\tau^{-1} n_k
\sum_{|u|=n_k}e^{-2S(u)}H(e^{S(u)}n_k^{-1/2})~\sim~\tau^{-1}n_k
\sum_{|u|=n_k}e^{-2S(u)}(S(u)-2^{-1}\log n_k)\quad \text{a.s.}$$
Here, the second equivalence is a consequence of \eqref{Hlog}. In
view of \eqref{x4}, $\lim_{k\to\infty}J_2(n_k)=0$ a.s. Passing to
the analysis of $J_1(n_k)$ we first note that, for each $\tau>0$
and $u$ with $|u|=n_k$,
\begin{equation}\label{intermed}
\lim_{k\to\infty} \big(H\big(\tau
e^{S(u)}n_k^{-1/2}+H(e^{S(u)}n_k^{-1/2})\big)-H\big(e^{S(u)}n_k^{-1/2}\big)\big)=\log\tau\quad\text{a.s.}
\end{equation}
Indeed, $H\big(\tau
e^{S(u)}n_k^{-1/2}+H(e^{S(u)}n_k^{-1/2})\big)-H\big(e^{S(u)}n_k^{-1/2}\big)\geq
H\big(\tau e^{S(u)}n_k^{-1/2}\big)-H\big(e^{S(u)}n_k^{-1/2}\big)$,
and recalling \eqref{x2}, the right-hand side converges to $\log
\tau$ a.s.\ as $k\to\infty$ by \eqref{x6}. In the converse
direction, observe that, in view of \eqref{Hlog},
$\lim_{t\to\infty}t^{-1}H(t)=0$. This in combination with
\eqref{x2} ensures that given $\delta>0$ we have, for large enough
$k$, that $$H\big(\tau
e^{S(u)}n_k^{-1/2}+H(e^{S(u)}n_k^{-1/2})\big)-H\big(e^{S(u)}n_k^{-1/2}\big)\leq
H\big((\tau+\delta)
e^{S(u)}n_k^{-1/2}\big)-H\big(e^{S(u)}n_k^{-1/2}\big).$$ Using
\eqref{x6} and sending first $k\to\infty$ and then $\delta\to 0+$
we conclude that the limit superior in \eqref{intermed} does not
exceed $\log \tau$. This completes the proof of \eqref{intermed}.
Invoking now \eqref{intermed} and \eqref{x2} yields, as
$k\to\infty$
$$J_1(n_k)~\sim~ (\log \tau) n_k^{1/2}
\sum_{|u|=n_k}e^{-S(u)}\quad\text{a.s.}$$ Hence, by \eqref{x3},
$\lim_{k\to\infty}J_1(n_k)=(\log\tau) (2/(\pi\sigma^2))^{1/2}Z$ a.s. The proof of \eqref{limit_main000} and \eqref{limit_main} is complete.

Assume now that Conditions $\mathcal{S}_{\rm na}$ and $\mathcal{S}^\ast$ hold. Put $\tilde H(x):=H(x)-\log x$ for $x>0$. By Theorem \ref{expa}, $\lim_{x\to\infty}\tilde H(x)=c$. This in combination with Lemma \ref{aux_limit} (a,b) ensures that $$n^{1/2}\sum_{|u|=n}e^{-S(u)}\tilde H(e^{S(u)}n^{-1/2})~\overset{\mmp}{\to}~(2/(\pi\sigma^2))^{1/2}cZ,\quad n\to\infty.$$ A minor modification of the proof of \eqref{eq:a} which takes into account the last limit relation justifies \eqref{limit_main01} and \eqref{limit_inter2}.

\section{Appendix}

\subsection{A link between a distribution tail and the Laplace transform}

In this section we give two results which connect the asymptotic behavior of a distribution tail at $\infty$ with that of the corresponding Laplace-Stieltjes transform at $0$.
\begin{lemma}\label{link1}
Let $b>0$ and $X$ be a nonnegative random variable with Laplace
transform $\phi^\ast(s):=\me e^{-sX}$ for $s\geq 0$. For $s>0$, set
$\psi^\ast(s):=s^{-1}(1-\varphi^\ast(s))$,
$$G^\ast(s):=\int_0^s \mmp\{X>y\}{\rm d}y\quad\text{and}\quad H^\ast(s):=\me
X\1_{\{X\leq s\}}.$$ The following assertions are equivalent:

\noindent (i) $\lim_{t\to\infty}t\mmp\{X>t\}=b$;

\noindent (ii) for each $\lambda>0$, $\lim_{s\to
0+}(\psi^\ast(s/\lambda)-\psi^\ast(s))=b\log \lambda$;

\noindent (iii) for each $\lambda>0$, $\lim_{t\to\infty}(G^\ast(\lambda
t)-G^\ast(t))=b\log \lambda$;

\noindent (iv) for each $\lambda>0$, $\lim_{t\to\infty}(H^\ast(\lambda
t)-H^\ast(t))=b\log \lambda$.

Either of these entails
\begin{equation}\label{limit_h}
\psi^\ast(1/t)~\sim~G^\ast(t)~\sim~ H^\ast(t)~\sim~ b\log t,\quad t\to\infty.
\end{equation}
\end{lemma}
\begin{remark}
Recall that functions $\psi^\ast$, $G^\ast$ and $H^\ast$ satisfying the
assumptions (ii), (iii) and (iv) of Lemma \ref{link1} belong to the de Haan class. In
particular, these functions are slowly varying ($\psi^\ast$ at zero,
$G^\ast$ and $H^\ast$ at $\infty$). Relation \eqref{limit_h} makes the last statement even more precise, showing that all these functions are asymptotically equivalent to the logarithm.
\end{remark}
\begin{proof}
The equivalence (i) $\Leftrightarrow$ (iii) follows from Theorem
3.6.8 in \cite{Bingham+Goldie+Teugels:1989}. The equivalence (ii)
$\Leftrightarrow$ (iii) follows from Theorem 3.9.1 in
\cite{Bingham+Goldie+Teugels:1989} after noting that
\begin{equation}\label{psi}
\psi^\ast(s)=\int_{[0,\,\infty)}e^{-sy}{\rm d}G^\ast(y),\quad s>0.
\end{equation}

\noindent {\sc Proof of} (iii) $\Rightarrow$ (iv). Integration by
parts yields
\begin{equation}\label{equ}
H^\ast(t)=\int_{[0,\,t]}y{\rm d}\mmp\{X\leq y\}=\int_0^t\mmp\{X>y\}{\rm
d}y-t\mmp\{X>t\}=G^\ast(t)-t\mmp\{X>t\}.
\end{equation}
According to the equivalence (i) $\Leftrightarrow$ (iii),
$\lim_{t\to\infty}t\mmp\{X>t\}=b$. Hence, invoking (iii) we arrive
at (iv).

\noindent {\sc Proof of} (iv) $\Rightarrow$ (i). Write, for any
$\delta>1$,
\begin{multline*}
t\mmp\{X>t\}=t\sum_{n\geq 1}\mmp\{t\delta^{n-1}<X\leq
t\delta^n\}\geq \sum_{n\geq 1}\delta^{-n} \int_{(t\delta^{n-1},\,
t\delta^n]}y{\rm d}\mmp\{X\leq y\}\\=\sum_{n\geq
1}\delta^{-n}(H^\ast(t\delta^n)-H^\ast(t\delta^{n-1})).
\end{multline*}
Relation (iv) entails that given $\varepsilon>0$
$$H^\ast(t\delta^n)-H^\ast(t\delta^{n-1})\geq b\log\delta-\varepsilon$$ for large enough $t$,
whence, for such $t$, $$t\mmp\{X>t\}\geq
(b\log\delta-\varepsilon))\sum_{n\geq
1}\delta^{-n}=(\delta-1)^{-1}(b\log\delta-\varepsilon).$$ Sending
first $\varepsilon\to 0+$ and then $\delta\to 1-$ we obtain
${\lim\inf}_{t\to\infty} t\mmp\{X>t\}\geq b$. A symmetric argument
proves the converse inequality for the limit superior.

Further, it is trivial that (i) entails $G^\ast(t)\sim b\log t$ as $t\to\infty$. With this at hand, $H^\ast(t)\sim b\log t$ as $t\to\infty$ is a consequence of \eqref{equ} and (i). Finally, $\psi^\ast(1/t)\sim G^\ast(t)$ as $t\to\infty$ follows from (ii) (or (iii)) and Theorem 3.9.1 in \cite{Bingham+Goldie+Teugels:1989}.
\end{proof}

\begin{lemma}\label{link}
Let $b>0$, $c\in\mr$ and $X$ be a nonnegative random variable with
Laplace transform $\phi^\ast$. The following assertions are
equivalent:

\noindent (I) $\psi^\ast(s)=s^{-1}(1-\phi^\ast(s))=-b\log s-\gamma+c+o(1)$ as $s\to 0+$, where $\gamma$ is the
Euler-Mascheroni constant;

\noindent (II) $G^\ast(t)=\int_0^t \mmp\{X>y\}{\rm d}y=b \log t+c+o(1)$ as $t\to\infty$;

\noindent (III) $H^\ast(t)=\me X\1_{\{X\leq t\}}=b \log t+c-b+o(1)$ as $t\to\infty$.
\end{lemma}
\begin{proof}
{\sc Proof of (I) $\Leftrightarrow$ (II)}. Let $\lambda>0$. Condition (I) ensures that $\lim_{s\to 0+}(\psi^\ast(s/\lambda)-\psi^\ast(s))=b\log
\lambda$. Recalling \eqref{psi} and invoking Theorem 3.9.1 in
\cite{Bingham+Goldie+Teugels:1989} we infer
$G^\ast(t)=\psi^\ast(1/t)+\gamma+o(1)$ as $t\to\infty$. This in conjunction
with (I) proves (II). In the converse direction, condition
(II) guarantees that $\lim_{t\to \infty}(G^\ast(\lambda t)-G^\ast(t))=b\log \lambda$. Another appeal to
Theorem 3.9.1 in \cite{Bingham+Goldie+Teugels:1989} allows us to
conclude that $\psi^\ast(s)=G^\ast(1/s)-\gamma+o(1)$ as $s\to 0+$, whence
(I).

\noindent {\sc Proof of (II) $\Rightarrow$ (III)}. As a consequence of (II), for each $\lambda>0$,
$\lim_{t\to\infty}(G^\ast(\lambda t)-G^\ast(t))=b\log \lambda$. Hence, $\lim_{t\to\infty}t\mmp\{X>t\}=b$ by the implication (iii)
$\Rightarrow$ (i) of Lemma \ref{link1}. With this, \eqref{equ} and
(II) at hand we obtain
$$H^\ast(t)=G^\ast(t)-t\mmp\{X>t\}=b\log t+c-b+o(1),\quad t\to\infty.$$

\noindent {\sc Proof of (III) $\Rightarrow$ (II)}. Relation
(III) implies that, for each $\lambda>0$, $\lim_{t\to\infty}(H^\ast(\lambda t)-H^\ast(t))=b \log \lambda$. Hence, by the
implication (iv) $\Rightarrow$ (i) of Lemma \ref{link1},
$\lim_{t\to\infty}t\mmp\{X>t\}=b$. Now (III) together with
\eqref{equ} ensures (II).
\end{proof}

\subsection{Results on standard random walks}
Here are some general results on the renewal functions associated
to the ascending or descending ladder height processes of a
centered random walk with finite variance.
\begin{lemma}\label{don}
Let $(T_n)_{n\in\mn_0}$ be a standard random walk with $T_0=0$,
$\me T_1=0$ and $\me T_1^2\in (0,\infty)$. Further, let
$\tau^\prime_-$ and $\tau^\prime_+$ denote a strictly or weakly
descending and a strictly or weakly ascending ladder epoch for
$(T_n)_{n\in\mn_0}$ and $\bar U$ the renewal function for the
standard random walk with jumps having the same distribution as
$|T_{\tau^\prime_-}|$ or $T_{\tau^\prime_+}$. Then

\noindent (a) $\me |T_{\tau^\prime_\pm}|<\infty$; for $\beta>2$,
$\me (T_1)_-^\beta<\infty$ is equivalent to $\me
|T_{\tau^\prime_-}|^{\beta-1}<\infty$ and $\me
(T_1)^\beta_+<\infty$ is equivalent to $\me T_{\tau^\prime_+}^{\beta-1}<\infty$;

\noindent (b) $\lim_{x\to\infty}x^{-1}\bar U(x)=(\me
|T_{\tau^\prime_\pm}|)^{-1}$.

\noindent (c) If the distribution of $T_1$ is
nonarithmetic/$d$-arithmetic for $d>0$, then so is the
distribution of $T_{\tau^\prime_\pm}$.

\noindent (d) Assume that the distribution of $T_1$ is nonarithmetic. Then $$\lim_{x\to\infty}((\me|T_{\tau^\prime_\pm}|)\bar U(x)-x)=c$$ for a finite constant $c$ if, and only if, $\me (T_1)_\pm^3<\infty$. If it is the case, then $c=(2\me |T_{\tau^\prime_\pm}|)^{-1}\me T_{\tau^\prime_\pm}^2$.
\end{lemma}

\begin{proof}
Part (a) is formula (4a) and Corollary 1 in \cite{Doney:1980}.
Part (b) is the elementary renewal theorem. For part (c), see, for
instance, p.~2156 in \cite{Bertoin+Doney:1994}. For part (d), first observe that $\me (T_1)_\pm^3<\infty$ is equivalent to $\me (T_{\tau^\prime_\pm})^2<\infty$. Now the result can be derived directly from the Blackwell theorem. Alternatively, while sufficiency of $\me (T_{\tau^\prime_\pm})^2<\infty$ follows from Example 3.10.3 on p.~242 in \cite{Resnick:2005}, necessity of that condition can be obtained along the lines of the aforementioned example with the help of Theorem 4 in \cite{Sgibnev:1981}.
\end{proof}

\subsection{Results on Lebesgue integrable and directly Riemann integrable functions}\label{sect:dri}

A function $t: \mr^+\to\mr^+$ is called {\it directly Riemann integrable} (dRi) on $\mr^+$, if

\noindent (a) $\overline{\sigma}(h)<\infty$ for each $h>0$ and

\noindent (b) $\lim_{h\to 0+} \big(\overline{\sigma}(h)-\underline{\sigma}(h)\big)=0$,
where $$\overline{\sigma}(h):=h\sum_{n\geq 1}\sup_{(n-1)h\leq
y<nh} t(y)\quad \text{and}\quad
\underline{\sigma}(h):=h\sum_{n\geq 1}\inf_{(n-1)h\leq
y<nh}t(y).$$ If $t$ is dRi, then $\lim_{h\to 0+}\overline{\sigma}(h)=\int_0^\infty t(y){\rm d}y<\infty$, where the integral is an improper Riemann integral.

Lemma \ref{dri} is concerned with an important step in the proof of Theorem \ref{impo}.
\begin{lemma}\label{dri}
Assume that $\underline{\sigma}(h_0)<\infty$ for some $h_0>0$ and that, for some $a\geq 0$, $x\mapsto e^{-ax}t(x)$ is a nonincreasing function on $\mr^+$. Then $t$ is dRi on $\mr^+$.
\end{lemma}
\begin{remark}
Lemma \ref{dri} is a strengthening of the well-known fact (see, for instance, Corollary 2.17 in \cite{Durrett+Liggett:1983}) that $t$ is dRi provided that $t$ is Lebesgue integrable and $x\mapsto e^{-ax}t(x)$ is a nonincreasing function. In Lemma \ref{dri} we require less, namely that $\underline{\sigma}(h_0)<\infty$ for some $h_0>0$ which is of course true if $t$ is Lebesgue integrable.
\end{remark}
\begin{proof}
Using twice the assumed monotonicity we obtain
\begin{multline*}\qquad \infty>e^{2ah_0}\underline{\sigma}(h_0)=e^{2ah_0}h_0\sum_{n\geq 1}\inf_{(n-1)h_0\leq
y<nh_0}(e^{ay} e^{-ay}t(y))\\
\geq e^{ah_0}h_0 \sum_{n\geq 1}t(nh_0)\geq \overline{\sigma}(h_0)-h_0\sup_{0\leq y<h_0}\,t(y).\qquad
\end{multline*}
This shows that $\overline{\sigma}(h_0)<\infty$. Remark 2.9 in \cite{Tsirelson:2013} states that, for $h>0$, $$\overline{\sigma}(h)\leq (1+2h/h_0)\overline{\sigma}(h_0)$$ which implies that $\overline{\sigma}(h)<\infty$ for each $h>0$. Hence, also $\underline{\sigma}(h)<\infty$ for each $h>0$ because $\underline{\sigma}(h)\leq \overline{\sigma}(h)$. Repeating now, for each $h>0$, the argument based on monotonicity we conclude that, for each $h>0$,
\begin{equation}\label{aux100}
e^{2ah}\underline{\sigma}(h)\geq \overline{\sigma}(h)-h\sup_{0\leq y<h}\,t(y).
\end{equation}
In view of $$\underline{\sigma}(h)\leq I:=\int_0^\infty t(y){\rm d}y\leq \overline{\sigma}(h) <\infty,\quad h>0,$$ we conclude that $t$ is Lebesgue integrable and that $\lim\sup_{h\to 0+}\underline{\sigma}(h)\leq I$, whence $$\lim_{h\to 0+}\underline{\sigma}(h)(e^{2ah}-1)=0.$$ Noting that $\lim_{h\to 0+}h\sup_{0\leq y<h}\,t(y)=0$ an appeal to \eqref{aux100} reveals that $\lim\sup_{h\to 0+}(\overline{\sigma}(h)-\underline{\sigma}(h))\leq 0$ which completes the proof.
\end{proof}

Lemma \ref{integr} is needed to justify statements made in Remark \ref{rem_im}.

\begin{lemma}\label{integr}
Let $t:\mr^+\to\mr^+$ and $V^\ast$ be the right-continuous renewal function of a standard random walk $(S_n^\ast)_{n\in\mn_0}$
with nonnegative jumps of finite mean $\mu^\ast$ which have a
nonarithmetic distribution.

\noindent (a) There exist improperly Riemann integrable $t$ and the renewal functions $V^\ast$ such that
\begin{equation}\label{integr1}
\int_{[0,\,\infty)}t(x+y){\rm d}V^\ast(y)<\infty
\end{equation}
fails to hold for each $x\geq 0$.

\noindent (b) If $t$ is dRi on $\mr^+$, then \eqref{integr1} holds for each $x\geq 0$.

\noindent (c) There exist continuous $t$ and the renewal functions $V^\ast$ such that \eqref{integr1} holds for some $x\geq 0$,
yet $\int_0^\infty t(y){\rm d}y=\infty$.

\noindent (d) Assume that
\begin{equation}\label{integr2}
\sup_{x\geq 0} \int_{[0,\,\infty)}t(x+y){\rm d}V^\ast(y)<\infty.
\end{equation}
Then $\int_0^\infty t(y){\rm d}y<\infty$.
\end{lemma}
\begin{proof}
(a) We only consider the case $x=0$. A modification needed to treat the case $x>0$ is obvious.  We use the same $t$ and $V^\ast$ as in Example 3.10.2 on p.~233 of \cite{Resnick:2005} designed to demonstrate that the key renewal theorem can fail for integrands which are not dRi.

Let a random variable $S_1^\ast$ take two values $\alpha$ and $1-\alpha$ for some {\it irrational} $\alpha\in (0,1)$.
Then the distribution of $S_1^\ast$ is nonlattice, and the renewal function $V^\ast$ is piecewise constant with jumps at the points of the form
$k_1\alpha+k_2(1-\alpha)$, $k_1, k_2\in\mn_0$, $k_1+k_2>0$. Arrange these points in increasing order and denote the resulting configuration by $b_1<b_2<\ldots$ Consider an infinite sequence of isosceles triangles which do not overlap. They are located in $\mr^+\times \mr^+$ and have bases situated on the $x$-axis. The triangles are enumerated $1$, $2,\ldots$ from left to right. The base of the $n$th triangle is centered at $b_n$ and has length $s_n$; the height of the $n$th triangle is equal to $1$. The sequence $(s_n)_{n\in\mn}$ is assumed to satisfy $\sum_{n\geq 1}s_n<\infty$.  Define the function $t$ as follows:
while $t(x)=0$ for $x$ which do not belong to the bases of the triangles, its graph passes through the equal sides of the triangles for all the other $x$. Plainly,
$\int_0^\infty t(x){\rm d}x=2^{-1}\sum_{n\geq 1}s_n<\infty$, that
is, the area of the region between the graph $y=t(x)$ and the
$x$-axis is equal to the sum of the areas of all the triangles. Thus,
$f$ is improperly Riemann integrable on $\mr^+$. Finally, since $t(b_n)=1$ for $n\in\mn$,
\begin{equation*}
\int_{[0,\,\infty)}t(x){\rm
d}V^\ast(x)=\sum_{n\geq 1}t(b_n)\sum_{k\geq 1}\mmp\{S^\ast_k=b_n\}=\sum_{n\geq 1}\sum_{k\geq 1}\mmp\{S^\ast_k=b_n\}=\infty.
\end{equation*}

\noindent (b) This follows from $$\int_{[0,\,\infty)}t(x+y){\rm d}V^\ast(y)\leq V^\ast(1)\sum_{n\geq \lfloor x\rfloor+1}\sup_{n-1\leq y<n}t(y)<\infty,\quad x\geq 0$$ which is just \eqref{dri3} with $t$ and $V^\ast$ replacing $g$ and $V$, respectively.

\noindent (c) We use the same $V^\ast$ as in part (a). To construct $t$, consider an infinite sequence of isosceles triangles enumerated $1$, $2,\ldots$ from left to right. They are located in $\mr^+\times \mr^+$, have heights $1$, and the endpoints of the base of the $n$th triangle are $b_n$ and $b_{n+1}$. Now we put $t(x)=0$ for $x<b_1$ and require that
the graph of $y=t(x)$ passes through the equal sides of the triangles for all the other positive $x$. Then $\int_0^\infty t(x){\rm d}x=(1/2)\lim_{n\to\infty}(b_n-b_1)=\infty$. On the other hand, $\int_{[0,\,\infty)}t(y){\rm d}V^\ast(y)=0$, so that \eqref{integr1} holds with $x=0$.

\noindent (d) Let $\xi_0^\ast$ be a random variable independent of $(S^\ast_n)_{n\geq 0}$ with distribution function $\mmp\{\xi_0^\ast\leq x\}=(1/\mu^\ast)\int_0^x\mmp\{S_1^\ast>y\}{\rm d}y$. On the one hand,
\begin{align*}
  \me\sum_{n\geq 0}t(\xi_0^\ast+S_n^\ast) &=\int_{[0,\,\infty)}\me\sum_{n\geq 0}t(x+S_n^\ast){\rm d}\mmp\{\xi_0^\ast\leq x\}\\
  &=\int_{[0,\,\infty)}\int_{[0,\,\infty)}t(x+y){\rm d}V^\ast(y){\rm d}\mmp\{\xi_0^\ast\leq x\}<\infty,
\end{align*}
where the finiteness is secured by \eqref{integr2}. On the other hand, the random process $(\tilde{N}^\ast(x))_{x\geq 0}$ defined by
\begin{equation*}
\tilde{N}^\ast(x) := \#\{n \in \mn_0: \xi_0^\ast+ S_n^\ast\leq x\},\quad x\geq 0
\end{equation*}
is a stationary renewal process (the term is standard but misleading; actually the process has stationary increments) which particularly implies that $\tilde{N}^\ast(x)=x/\mu^\ast$ for $x\geq 0$. Hence, $\int_0^\infty t(y){\rm d}y= \me\sum_{n\geq 0}t(\xi_0^\ast+S_n^\ast)<\infty$.
\end{proof}

\vspace{5mm}

\noindent {\bf Acknowledgement}. The authors thank the organizers
and participants of the conference `Branching in Innsbruck'
(September 2019). In particular, the talk `1-stable fluctuations
of branching Brownian motion at critical temperature' delivered by
Pascal Maillard during that conference has given a major impetus
to the research reported here. This project was initiated while D.B.\ and A.I.\ were visiting LAGA - Institut Galilée, Université Sorbonne Paris Nord in October 2019. Both gratefully acknowledge hospitality and the financial support. Also, D.B.\ and A.I.\ thank Alexander Marynych for a useful discussion. A major part of A.I.'s work was done under the support of Ulam programme funded by the Polish national agency for academic exchange (NAWA), project no. PPN/ULM/2019/1/00010/DEC/1. D.B. was partially supported by the National Science Center,
Poland (Sonata Bis, grant number DEC-2014/14/E/ST1/00588).

\end{document}